\documentclass[twoside,12pt]{amsart}
\usepackage{amsmath,amsfonts}
\usepackage[dvips]{graphicx}
\usepackage{color}
\usepackage{amsthm}
\usepackage{amsfonts}
\usepackage{amssymb}
\usepackage{latexsym}
\numberwithin{equation}{section}

\makeatletter
\def\serieslogo@{}
\makeatother \makeatletter
\def\@setcopyright{}
\makeatother

\newcommand\ds{\displaystyle}
 

\usepackage{color}

\def\p{\partial}

\def\dg{\delta^\gamma}

 \def\R{{\mathbb R}}
 
 \def\dis{\displaystyle}

 \def\F{\mathcal{F}}
  \def\H{{\oc{\mathcal{H}}}}
  \def\Ht{{\mathcal{H}}}
\def\oc{\overset\circ}
\def\Hg{{\oc H{}^1_\mu}}
\def\Hgo{{\oc H{}^1_{\mu_0}}}

\def\nm{\nabla_m}

\def\we{\ln \frac e\rho}
\def\wet{\ln^2 \frac e\rho}

\def\ep{\varepsilon}

\newtheorem{thm}{Theorem}
\newtheorem{lem}[thm]{Lemma}

\newtheorem{prop}[thm]{Proposition}

\newtheorem{defn}[thm]{Definition}
\newtheorem{rmk}[thm]{Remark}

\title[]{The Cauchy-Dirichlet problem for the FENE Dumbbell model of polymeric
fluids}
\author{
Hailiang Liu and Jaemin Shin
}
\address{Iowa State University, Mathematics Department, Ames, IA 50011} \email{hliu@iastate.edu}
\address{Institute for Mathematics and its applications,
University of Minnesota, Minneapolis, MN 55455,
} \email{shin@ima.umn.edu}

\date{September 29, 2010}

\begin{document}

\begin{abstract}
The FENE dumbbell model consists of the incompressible Navier-Stokes
equation and the Fokker-Planck equation for the polymer distribution.
In such a model, the polymer elongation cannot exceed
a limit $\sqrt{b}$, yielding all interesting features near the
boundary. In this paper we establish the local well-posedness for
the FENE dumbbell model under a class of Dirichlet-type boundary
conditions dictated by the parameter $b$. As a result, for each $b>0$
we identify a sharp boundary requirement for the underlying density distribution,
while the sharpness follows from the existence result for each
specification of the boundary behavior.  
It is shown that the probability density governed by the Fokker-Planck equation  approaches  zero near boundary,
necessarily faster than the distance function $d$ for $b>2$, faster than $d|ln d|$ for $b=2$,  and as fast as $d^{b/2}$ for $0<b<2$.
Moreover, the sharp boundary requirement for $b\geq 2$ is also sufficient for the distribution to remain
a probability density.

\end{abstract}
\maketitle

\bigskip


\section{Introduction}
Let $N\geq 2$ be an integer.  We consider a dimer -- an idealized
polymer chain -- as an elastic dumbbell consisting of two beads
joined by a spring that can be modeled by an elongation  vector
$m\in \R^N$ (see e.g \cite{BCAH87}), with $\Psi$ being the elastic
spring potential defined by
\begin{equation}\label{potential}
\Psi(m)= -\frac{Hb}{2}\log \left(1-\frac{|m|^2}{b}\right), \quad
m\in B.
\end{equation}
Here $B:=B(0, \sqrt{b})$ is a ball in $\R^N$ with radius $\sqrt{b}$
denoting the maximum dumbbell extension. In the limiting case, this
reduces to the Hookean model with $\Psi(m)=H|m|^2/2$. A general
bead-spring chain model may contain more than two beads coupled with
elastic springs to represent a polymer chain.

Polymers as such when put into an incompressible, viscous, isothermal Newtonian solvent are modeled by a
system coupling the incompressible Navier-Stokes equation for the
macroscopic velocity field $v(t,x)$ with the Fokker-Planck equation
for the probability distribution function $f(t,x,m):$
\begin{eqnarray}
\p_t v + (v\cdot \nabla) v + \nabla p &=& \nabla \cdot \tau + \nu_k \Delta v, \label{n1}\\
\nabla\cdot v&=&0,\label{n2}\\
\p_tf+ (v\cdot \nabla)f+ \nabla_m\cdot(\nabla v m f)&=&\frac2\zeta
\nabla_m \cdot(\nabla_m \Psi(m)f)+\frac{2k_BT_a}{\zeta}\Delta_m
f,\label{n3}
\end{eqnarray}
where $x\in\R^N$ is the macroscopic Eulerian coordinate and $m\in
B\subset \R^N$ is the microscopic molecular configuration variable.
The model describes diluted solutions of polymeric liquids with
noninteracting polymer chains (dimers). Note that the Fokker-Planck equation can be
conveniently augmented to incorporate other effects such as inertial
forces (see \cite{DL09}).

In  Navier-Stokes equation \eqref{n1}, $p$ is hydrostatic
pressure, $\nu_k$ is the kinematic viscosity coefficient, and $\tau$
is a tensor representing the polymer contribution to stress,
\begin{equation*}
\tau=\lambda_p \int m\otimes\nabla_m \Psi(m)f dm,
\end{equation*}
where  $\lambda_p$ is the polymer density constant.  In the
Fokker-Planck equation \eqref{n3}, $\zeta$ is the friction
coefficient of the dumbbell beads, $T_a$ is the absolute
temperature, and $k_B$ is the Boltzmann constant. We refer
to \cite{BCAH87, DE:1986, Oe:1996} for a comprehensive survey of the physical background, and
\cite{OP02} for the computational aspect.

Since $B$ is a bounded domain, one has to add an appropriate
boundary condition for $f$ on the boundary $\partial B$. However,
the singularity of the Fokker-Planck equation near $\partial B$
makes the boundary issue rather subtle, and presents various
challenges. To address the boundary issue, several transformations
relating to the equilibrium solution have been introduced in
literature ( see, e.g. \cite{DLY05, KS09, LiLi08, LiSh08}). It was
shown in  \cite{LiLi08} that $b=2$ is a threshold in the sense that
for $b\geq 2$ any preassigned boundary value of the ratio of the
distribution and the equilibrium will become redundant, and for
$b<2$ that value has to be a priori given.

Our main quest in this paper is that what is the
least boundary requirement for $f$ so that both existence and uniqueness
of solutions to the FENE model can be established, also the solution remains a probability density.

We addressed this issue in \cite{LiSh08} for the microscopic FENE
model alone and when  $b>2$.  In this article we consider the
well-posedness of the coupled system (\ref{n1})-(\ref{n3}).  A
general discussion of this problem and background references are
given in the introduction to \cite{LiSh08}.  Here we have two
objectives:
\begin{itemize}
\item[(1)] to identify sharp boundary conditions on $\partial B$ for all  $b>0$.
\item[ (2)] to prove well-posedness  for the coupled FENE dumbbell model under the identified boundary condition.
\end{itemize}
The setting for our problem is the coupled system subject to the
initial data
\begin{eqnarray}
v(0,x)&=& v_0(x)\\
f(0,x,m)&=& f_0(x,m),
\end{eqnarray}
with the following boundary requirement
\begin{equation}\label{ft}
f(t,x,m)\nu^{-1}|_{\p B}= q(t,x,m)|_{\p B}.
\end{equation}
Here $\nu$ depends on $b$ through the distance function, and $q$ is a given
function measuring the relative ratio of $f/\nu$ near boundary. Our
goal is to investigate solvability of the above system with the
Cauchy-Dirichlet data. Note that our boundary condition is
more or less a boundary behavior requirement for $f$, instead of the
Dirichlet data in the traditional  sense.

Instead of using the distance function $d=\sqrt{b}-|m|$ we shall use
a regularized distance function $\rho=b-|m|^2$ when describing the
solution  behavior near boundary. Our main observation is  the form
of $\nu$
\begin{equation}\label{nu}
\nu=\left\{
      \begin{array}{ll}
        \rho^{b/2}, & 0<b<2, \\
        \rho \we, &  b=2, \\
        \rho, &  b>2.
      \end{array}
    \right.
\end{equation}
With some regularity requirement on $q$ as well as on initial data
we prove local well-posedness for the Cauchy-Direchlet problem in a
weighted Sobolev space for each given $q$. Our results indicate
that simply putting $f=0$ on boundary does not guarantee uniqueness
of the solution.

For the Dirichlet-type boundary condition (\ref{ft}) considered in this paper, our strategy is to study
the transformed
problem via
$$
w=\frac{f}{\nu} -q
$$
with $\nu$ defined in (\ref{nu}) so as to extract useful info for $f$.  Inspired from  \cite{Ma08},  for the coupled FENE system we
use weak norm in $m$ and strong norm in $x$, this enables us to prove wellposedness for all
cases of $b>0$ and any given smooth $q$.

For the case $b\geq 2$  of physical interest,  we prove that $f$
remains a density distribution if and only if  $q|_{\p B}=0$. We
thus identify a sharp boundary requirement for each $b>0$ for the
underlying density distribution, while the sharpness is a
consequence of the existence result for each $q\not=0$.  In
particular, our result asserts that near boundary the probability
density governed by the Fokker-Planck equation  approaches
zero, necessarily faster than the distance function $d$ for $b>2$,
faster than $d|ln
d|$ for $b=2$, and as fast as $d^{b/2}$ for $0<b<2$.  
But within our current framework we have not been able to identify  a non-trivial $q$ for $0<b<2$ such that the corresponding
solution is a density distribution.

We remark that the sharp boundary condition  presented in this work
provides a threshold on the boundary requirement:  subject to this
condition or any stronger ones incorporated through a weighted
function space \cite{ZhZh06} or just zero flux \cite{Ma08}, the
Fokker-Planck dynamics will select the physically relevant solution,
which is a probability density, any weaker boundary requirement can
lead to many solutions, each depending on the ratio of  $f/\nu$ near
boundary.

This article is organized as follows.  In Section 2, we state our
main results and mains ideas of the proofs.  In Section 3, we study
the Fokker-Planck operator and well-posedness of the initial
boundary value problem for the Fokker-Planck equation alone. This
improves upon our previous work in \cite{LiSh08}. The main result is
summarized in Theorem \ref{prop1}. The Fokker-Planck problem
involving spatial variable $x$ is investigated in Section 4.
Well-posedness of the coupled system is proved in Section 5. In
Section 6, we sketch the proof of well-posedness for the coupled system with $b\geq 6$ in a
different function space than what we used in Section 5.   Some concluding remarks are drawn in Section 7.\\

We conclude this section by some bibliographical remarks.\\

Existence results for the FENE model are usually limited to small-time existence
and uniqueness of strong solutions.  We refer to \cite{Rm91} for the local
existence on some related coupled systems, \cite{JLL04} for the FENE model (in the setting where the
Fokker-Planck equation is formulated by a stochastic differential
equation) with $b > 6$, \cite{ETiZh04} for a polynomial force. More related to this paper  are  the work by Zhang and Zhang \cite{ZhZh06} for the FENE model when $b >76$,  and
Masmoudi \cite{Ma08}  for $b>0$.  Global existence results are
usually limited to solutions near equilibrium,
see  \cite{LLZ07, LiZh08}, or to some 2D simplified models \cite{CFTZ07,
CM08,LiZhZh08, MZZ08}. For results concerning the existence of weak solutions to the coupled FENE system we refer to
\cite{BaSchSu05,BaSu07,BaSu08,BaSu10, LiMa07, Ma10, Sch09, ZZZ08}.

Boundary behavior of the polymer distribution governed by the FENE
model is also essential in several other aspects, including the
study of large time behavior, see \cite{ACM09,  HZ09, JLLO06,
Sch09}; and development of numerical methods, see, e.g., \cite{CL04,
CL04+, DLY05, KS09, LY10, SY10}. We also refer to \cite{JoLe03} for
references on numerical aspects of polymeric fluid models.

There are also some interesting works on non-Newtonian fluid models
derived through a closure of the linear Fokker-Planck equation (see,
e.g., \cite{DLP02, DLY05}). We can refer to the pioneering work
\cite{GS90, GS90+}, and more recently to \cite{CM01, LLZ05, LLZ08,
LM2000, LeZ05}.

However, none of these works is concerned with the sharpness of boundary conditions
in terms of the elongation parameter.

\section{Main results}\label{sec2}
After a suitable scaling and choice of parameters we arrive at  the
following Cauchy-Dirichlet probelm for the coupled system
\begin{subequations}\label{main1}
\begin{eqnarray}
\p_t v + (v\cdot \nabla) v + \nabla p &=& \nabla \cdot \tau + \Delta v, \quad x\in \R^N, \quad t>0,\\
\nabla\cdot v&=&0,\\
\p_tf+ (v\cdot \nabla)f+\nabla_m\cdot(\nabla v m f)&=&\frac12
\nabla_m
\cdot\left(\frac{bm}{\rho} f \right)+\frac12\Delta_m f, \quad m\in B, \\
 \tau&=& \int m\otimes\frac{bm}{\rho}f dm, \\
v(0,x)&=&v_0(x),\\
f(0,x,m)&=&f_0(x,m),  \\
f(t,x,m)\nu^{-1}|_{\p B}&=& q(t,x,m)|_{\p B}.\label{main_boundary}
\end{eqnarray}
\end{subequations}

To present our main results we first fix notations to be used
throughout  this article.  We fix an exponent $s$, which is an
integer in the range  $s >N/2+1$.  We use $C$ to denote various
constants depending on $s$, $b$  and on some other quantities which
we will indicate in the sequel.  A $b$-dependent weight function is
defined as
\begin{eqnarray}  \label{mu}
\mu=\left\{
      \begin{array}{ll}
        \rho^{b/2}, & 0<b<2, \\
        \rho \wet, &  b=2, \\
        \rho^{2-b/2}, & b>2.
      \end{array}
    \right.
      \end{eqnarray}
   For $b\geq 6$, we also use
    \begin{equation}
        \mu_0=\rho^\theta, ~-1<\theta<1,~ b\geq 6.
\end{equation}
Other notations are listed as below as well.
\begin{itemize}
  \item $\dis L^2_\mu =\left\{\phi:\int_B  \phi^2 \mu dm<\infty
\right\}$
  \item $\dis H^1_{\mu}=\{\phi: \phi, \p_{m_j}\phi \in L^2_\mu, j=1\cdots N.\}$
  \item $\overset\circ H{}^1_\mu$ denotes the completion of $C^\infty_c$ with $H^1_\mu$ norm.
  \item $H^*$ is a dual space of $H$
  \item $H^s_x$ is the usual Sobolev space with respect to $x$
   \item \begin{align*}
   |v|_s^2 &= \sum_{|\alpha|\leq s} \int_{\R^N}  |\p^\alpha v|^2 dx,\\
   |w|^2_{0,s} &=  \sum_{|\alpha|\leq s} \int_{\R^N}  \int_B |\p^\alpha w|^2 \mu dm dx,\\
|w|^2_{1,s} &= |w|^2_{0,s} + |\nabla_m w|^2_{0,s},\\
\|w\|^2_{1, 1, s}& =\sup_t (|w|^2_{1, s}+|\partial_t w|_{1, s}^2),\\
||q||& =  \|q\|_{H^1_\mu}+ \|\p_t q\|_{H^1_\mu}.
\end{align*}
 \item $H^s_x L^2_\mu =\{\phi: |\phi|_{0,s} < \infty\}, \quad H^s_xH^1_\mu =\{\phi: |\phi|_{1,s} < \infty\}.$
 \item $L^2_t H =L^2((0,T); H), \quad C_tH=C([0,T];H)$ for $0<t<T.$
 \item $\dis \mathcal{H} =\{\phi: ||\phi||_{L^2_t H^1_\mu} +
||\phi_t||_{L^2_t (H^1_\mu)^*}< \infty\}$\\
 $\dis \H =\{\phi(t,\cdot)\in \Hg: ||\phi||_{L^2_tH^1_\mu} +
||\phi_t||_{L^2_t(\Hg)^*}< \infty\}.$
\item
$$
\dis \mathbf{X}_\mu= [C_t H^s_x \cap L^2_t H_x^{s+1}]\times
[C_tH^s_xL^2_\mu \cap L^2_tH^s_xH^1_\mu].
$$
\item  For a generic constant independent of $T$ and $a\in L^2_t$ we denote
\begin{equation}\label{F}
F(a)=C \left(T+\int_0^T |a(t)|^2 dt \right).
 \end{equation}
 Due to such a constant, any two instances of $F$ should be presumed to be with different constants.
\end{itemize}

We now state our main theorem as follows:

\begin{thm}\label{thm1}
Let $b>0$ and $s$ be an integer such that $s>N/2+1$. Suppose that
$v_0 \in H^s_x$, $f_0\nu^{-1} \in H^s_xL^2_{\mu}$, and $q \in
C^1_tH^{s+1}_x H^1_\mu$. Then, for some $T>0$ there exists a unique
solution $(v,f)$ to the coupled problem \eqref{main1} such that
$$
(v,f\nu^{-1})\in\mathbf{X}_\mu.
$$
\end{thm}

It is known from (\cite{Ku85}) that $\Hg = H^1_\mu$ for $b\geq6$ with $\mu$ defined in \eqref{mu}. Thus, the boundary condition
\eqref{main_boundary} is nothing but the zero dirichlet boundary
condition under the assumption on $q$ in Theorem \ref{thm1}. For
non-trivial $q$ when $b\geq 6$, we show the well-posedness in a
different weighted Sobolev space. The result summarized as below.

\begin{thm}\label{thm4}
Let $b\geq 6$ and $s$ be an integer such that $s>N/2+1$. Suppose
that $v_0 \in H^s_x$, $f_0\nu^{-1} \in H^s_xL^2_{\mu_0}$, and $q \in
C^1_tH^{s+1}_x H^1_{\mu_0}$ with $\mu_0$ defined in (\ref{mu}). Then, for some $T>0$ there exists a
unique solution $(v,f)$ to the coupled problem \eqref{main1} such
that
$$
(v,f\nu^{-1})\in \mathbf{X}_{\mu_0}.
$$
\end{thm}

Theorem \ref{thm1} and \ref{thm4}  tell us that for each given $q$,
which denotes the rate of $f$ approaching to zero relative to $\nu$
near $\partial B$, there exists a unique solution $(v, f)$. Also,
they indicate that any weaker boundary requirement may lead to  more
than one solutions to \eqref{main1}. For instance, the boundary
condition
\begin{equation*}
f\nu^{-1}\rho^{\ep}|_{\p B}=0, \quad \ep>0
\end{equation*}
gives infinitely many solutions to \eqref{main1}. Precisely we state
the following non-uniqueness result.

\begin{thm}\label{thm1.5}
Let $\tilde\nu$ be a smooth function of $\rho$ such that
\begin{equation}\label{Tilde_Nu}
\lim_{\rho \to 0} \frac{\nu}{\tilde\nu} =0.
\end{equation}
Then, the coupled problem \eqref{main1} with \eqref{main_boundary}
replaced by
\begin{equation}\label{another_bdy}
f(t,x,m){\tilde\nu}^{-1}|_{\p B}=0
\end{equation}
has infinitely many solutions in $\mathbf{X}_{\mu}$ and
$\mathbf{X}_{\mu_0}$ for $0<b<6$ and $b\geq 6$ respectively.
\end{thm}

The natural question is  for what $q$  the obtained distribution $f$
is a probability distribution.   The answer when $b\geq 2$ is given
in the following theorem.

\begin{thm}\label{thm2}  Suppose that $b\geq 2$ and $q|_{\p B} \geq
0$.  Under the assumption of Theorem \ref{thm1} or \ref{thm4}, the
unique solution $f$ to the Cauchy-Direchlet problem \eqref{main1} is
a probability distribution if and only if $q|_{\p B}=0$.  That is,
$f \geq 0$ if $f_0 \geq 0$,  and for any $t>0$, $x\in \R^N$,
\begin{equation}
\int_B f(t,x,m)dm = \int_B f_0 (x, m)dm.
\end{equation}
\end{thm}
Theorem \ref{thm1} is proven by a fixed point argument, which is now
outlined. Given $(u, g)$,  we first solve the Navier-Stokes equation
(NSE):
\begin{subequations}\label{NSE}
\begin{eqnarray}
\p_t v + (u\cdot \nabla) v + \nabla p &=& \nabla \cdot \tau + \Delta v, \label{NSE1}\\
\nabla\cdot v&=&0,\label{NSE2}\\
v(0,x)&=&v_0(x),\label{NSE3}\\
 \tau&=& \int m\otimes\frac{bm}{\rho}g dm. \label{NSE5}
\end{eqnarray}
 \end{subequations}
With the obtained $v$ we solve the Fokker-Planck equation (FPE):
\begin{subequations}\label{FPE}
\begin{eqnarray}
\p_tf+ (v\cdot \nabla)f+\nabla_m\cdot(\nabla v m f)&=&\frac12
\nabla_m
\cdot\left(\frac{bm}{\rho} f \right)+\frac12\Delta_m f,\label{FPE1}\\
f(0,x,m)&=&f_0(x,m), \label{FPE2} \\
f(t,x,m)\nu^{-1}|_{\p B}&=& q(t,x,m)|_{\p B}. \label{FPE3}
\end{eqnarray}
 \end{subequations}
The above two systems define a mapping $(u, g) \to (v, f)$, the
existence of problem (\ref{main1}) is equivalent to existence of a fixed point of this
mapping.

The main difficulty lies in monitoring the boundary behavior of
$f$. Our strategy is to apply the transformation
\begin{equation}\label{tr}
f= \nu(w+q),
\end{equation}
to (\ref{FPE}) to obtain a w-problem
\begin{subequations}\label{W}
\begin{eqnarray}
\mu (\p_t+ v\cdot \nabla) w +L[w]& =& \mu h, \label{w1}\\
w(0,x,m)&=&w_0(x,m), \label{w2} \\
w(t,x,m)|_{\p B}&=&0. \label{w3}
\end{eqnarray}
 \end{subequations}
Here the operator  $L$  is induced from the Fokker-Planck operator,
 $\nu$ and $\mu$ are weights depending on the distance functions
 defined in \eqref{nu} and \eqref{mu} respectively. The source  term is obtained from $q$
\begin{equation}\label{h}
h=-\partial_t q - (v\cdot \nabla) q-\mu^{-1}L[q],
\end{equation}
and the initial data  is given by
\begin{equation}\label{w0}
w_0(x, m):= f_0(x, m)\nu^{-1}-q(0, x, m).
\end{equation}
For given $(u,\varpi)$ with $g=\nu(\varpi+q)$,  we arrive at a map
$\F$.
\begin{eqnarray*}
\F : &\mathbf{M} & \to \mathbf{M}\\
     &(u,\varpi)&  \mapsto (v,w)
\end{eqnarray*}
 Here $\mathbf{M}$ is a subset of
 $$
 C_t H^s_x \times
[ C_tH^s_x L^2_\mu \cap L^2_t H^s_x \Hg]
$$
such that
\begin{eqnarray*}
\mathbf{M}&=& \left \{(v,w): \sup_{0\leq t\leq T} |v|^2_s \leq A_1,
\sup_{0\leq t\leq T} |w|^2_{0,s} + \frac12\int_0^{T} |\nm w|^2_{0,s} dt
\leq A_2 \right\}.
\end{eqnarray*}
The strategy for the fixed point proof, which we implement in
sections to follow,  is to first prove that $\mathcal F$ is well
defined for some $T, A_1$ and $A_2$,  then show that  $\F$ is actually a contraction map in a weak norm.
Moreover, we will show that
\begin{equation}\label{3.17}
\mathcal{F}(\mathbf{M}) \subset \mathbf{X}_\mu.
\end{equation}
This proves Theorem \ref{thm1} for
$$
q \in C^1_t H^{s+1}_x H^1_\mu \subset [C_t H^s_x L^2_\mu \cap
L^2_tH^s_xH^1_\mu].
$$

Theorem \ref{thm4} is proved in the same manner. A sketch of proof
is presented in Section 6.

In order to prove Theorem \ref{thm1.5}, we pick $q(t,x,\cdot) \in
C^\infty(B) \cap C(\overline{B})$ and $q|_{\p B}\neq 0$ such that
\begin{eqnarray*}
q\in \left\{
       \begin{array}{ll}
         C^1_tH^{s+1}_xH^1_\mu, & 0<b<6, \\
         C^1_tH^{s+1}_xH^1_{\mu_0}, & b\geq6.
       \end{array}
     \right.
\end{eqnarray*}
Note that existence of such a $q$ follows from the density of the
weighted Sobolev space (see \cite{Ku85} for details). Then for each
$q$ we have a unique solution $(v,f)$ to the coupled problem
\eqref{main1} from Theorem \ref{thm1} and Theorem \ref{thm4}. Now,
we check  the boundary condition \eqref{another_bdy}.
\begin{equation*}
f{\tilde\nu}^{-1}|_{\p B} = f\nu^{-1}\frac{\nu}{\tilde\nu}|_{\p B} =
q\frac{\nu}{\tilde \nu}|_{\p B},
\end{equation*}
which vanishes since $q|_{\p B}$ is bounded and  condition
\eqref{Tilde_Nu} holds. This proves Theorem \ref{thm1.5}.

Theorem \ref{thm2} follows from Proposition \ref{prop8} and
\ref{prop9} via a flow map to be described in Section
\ref{section4}. The case for $b\geq 6$ can be proved by a simple
modification, which is also sketched in Section 6.

\section{The Fokker-Planck operator}
We start with \eqref{FPE} when $x$ is not involved. In such a case it reduces to the following problem:
\begin{subequations} \label{ff}
\begin{eqnarray}
\p_tf+  \mathcal{L}[f]  &=& 0, \quad m\in B, t>0,  \label{LFPE1}\\
f(0,m)&=&f_0(m), \label{LFPE2} \\
f(t,m)\nu^{-1}|_{\p B}&=& q(t,m)|_{\p B}.\label{LFPE3}
\end{eqnarray}
\end{subequations}
Here
\begin{equation}\label{of}
\mathcal{L}[f] :=  \nabla\cdot(\kappa m f)- \frac12 \nabla
\cdot\left(\frac{bm}{\rho} f \right)- \frac12\Delta f,
\end{equation}
$\kappa=\kappa(t)$ is a square integrable 
matrix function such that
$\text{Tr}(\kappa)=0$. We omit $m$ from $\nabla_m$ in (\ref{of}) for notational
convenience.

The goal of this section is two folds;
\begin{enumerate}
  \item[(1)] to provide tools for subsequent sections.
  \item[(2)] to elaborate on this model alone as an extension of our
  previous work \cite{LiSh08}.
\end{enumerate}

\subsection{Transformed operator}\label{section3.1}

The transformation (\ref{tr}) leads to
\begin{subequations} \label{ww}
\begin{eqnarray}
 \p_t w\mu + L[w] &=& \mu h,  \quad m\in B, t>0, \label{4.4}\\
w(0,m) &=&w_0, \label{4.5}\\
w(t,m)|_{\p B} &=&0,  \label{4.6}
\end{eqnarray}
\end{subequations}
with the transformed operator $L$  determined by
\begin{equation}\label{lw0}
L[w]=\mu\nu^{-1} \mathcal{L}[\nu w].
\end{equation}
The source term  $h=-\partial_t q -\mu^{-1}L[q]$ and initial data for $w$ is $w_0=f_0\nu^{-1}-q(0,
m)$.

From a direct calculation with the choice of $\mu$ in (\ref{mu}),
and $\nu$  in \eqref{nu}, \eqref{lw0} can be expressed as
\begin{equation}\label{lw}
L[w]= -\frac12 \nabla \cdot (\nabla w \mu) + \nabla \cdot (\kappa m
w \mu) -Kw,
\end{equation}
where
\begin{equation}\label{K}
K=\left\{
      \begin{array}{ll}
        0, & 0<b<2, \\
        (N+2\kappa m\cdot m)  \dis \ln \frac{e}{\rho}, &  b=2, \\
        (N+2\kappa m\cdot m) (b/2-1) \rho^{1-b/2}, & b>2.
      \end{array}
    \right.
\end{equation}

Associated with  the operator $L$, we define its time-dependent
bilinear form
\begin{equation}\label{bb}
\mathcal{B}[w,\phi; t]:=  \int  \left(  \frac{1}{2}\nabla w\cdot \nabla \phi
\mu - w\mu \kappa m \cdot \nabla \phi - Kw\phi  \right)dm
\end{equation}
for $\phi, w \in \Hg$ and  fixed $t>0$.

We now describe the weak solution which we are looking for.

\begin{defn}\label{defn1}
A function $w \in \H$
is a weak solution of w-problem (\ref{ww}), provided
\begin{enumerate}
  \item[(1)] For each $\phi\in \Hg$ and almost every $0\leq t\leq
  T$,
  \begin{equation}\label{4.9}
(\p_t w, \phi)_{\Hg} + \mathcal{B}[w,\phi;t] = ( h, \phi)_{\Hg}.
\end{equation}
   \item[(2)] $\dis w(0,m)=w_0(m)$ in $L^2_\mu $ sense. i.e.
  \begin{equation*}\label{2.25}
   \int_{B} |w(0,m)-w_0(m)|^2 \mu dm =0.
  \end{equation*}
\end{enumerate}
\end{defn}
\begin{rmk}
In (\ref{4.9}),  $(\psi,\phi)_{\Hg}$ is a dual pair for $\psi\in
(\Hg)^*$ and $\phi\in \Hg$,  and  can be regarded as $L^2_\mu$ inner
product. Indeed, from the Riesz representaiton theorem, for each
$\psi\in (\Hg)^*$ there exists a unique $u \in \Hg$ such that
\begin{eqnarray*}
(\psi,\phi)_{\Hg}=\int_B (\nabla u \cdot \nabla \phi + u\phi) \mu dm.
\end{eqnarray*}
Formally, the right hand side will be
\begin{equation*}
\int_B  (\nabla\cdot (\nabla u\mu)\mu^{-1}+u ) \phi \mu dm.
\end{equation*}
We identify $\psi$ as $\nabla \cdot (\nabla u\mu)\mu^{-1}+u$ and the dual
pair will be the $L^2_\mu$ inner product.

\begin{rmk}
With the weight function $\mu$ so chosen as \eqref{mu}, we observe
that if $\phi \in H^1_\mu$, then  $\phi \in W^{1, 1}$ since
\begin{eqnarray*}
\int_B  ( |\phi| + |\nabla \phi| ) dm \leq C \left(\int_B (|\phi|^2
+ |\nabla \phi|^2)\mu dm \right)^{1/2} \left(\int_B \mu^{-1}
dm\right)^{1/2}< \infty.
\end{eqnarray*}
From the standard trace theorem,  the  map
 \begin{eqnarray*}
       \mathcal{T}: &H^1_{\mu}(B)& \to L^1(\p B)\\
          &\phi & \mapsto \phi|_{\p \Omega}
    \end{eqnarray*}
is well defined. Thus, the element in $\Hg$ is characterized by the
zero trace,  and the Dirichlet data \eqref{4.6} makes sense.
\end{rmk}
\end{rmk}

The well-posedness of the w-problem (\ref{ww}) is stated in the
following.

\begin{thm}\label{thm6}
Suppose that $w_0\in L^2_\mu$,  $h\in L^2_t(\Hg)^*$ and $\kappa \in
L^2_t$ with $\text{Tr}(\kappa)=0$.  Then the
w-problem (\ref{ww}) has a unique weak solution in $\H$ such that
\begin{equation}\label{w_estimate2}
||w||^2_\Ht \leq e^{F(|\kappa|)}(||w_0||^2_{L^2_\mu} + || h||^2_{L^2_t(\Hg)^*})
\end{equation}
with $F$ defined in (\ref{F}).
\end{thm}
This result when $b>2$ and $q=0$ was proved in \cite{LiSh08}. For
general case we proceed in several steps.

\subsection*{An embedding theorem}
We define
\begin{equation}\label{ms}
\mu^*=\left\{
      \begin{array}{ll}
        \rho^{b/2-2}, & 0<b<2, \\
        \rho^{-1}, &  b=2, \\
        \rho^{-b/2}, & b>2.
      \end{array}
    \right.
\end{equation}
We call $\mu^*$  as the conjugate of $\mu$ due to the Sobolev
inequalities in the following lemma.

\begin{lem}\label{lem1}
  If  $\phi \in \Hg $, then
    \begin{equation}\label{2.27}
       \int |\phi|^2 \mu^* dm \leq C \int (|\phi|^2+|\nabla \phi|^2)
\mu dm.
    \end{equation}
Also, if $\phi \in H^1_{\rho^{\theta}}$ for $\theta \leq 1$, then
for any $\delta>0$
\begin{equation}\label{2.27-1}
\int |\phi|^2 \rho^{-1+\delta}dm \leq C \int (|\phi|^2+|\nabla
\phi|^2) \rho^\theta dm.
\end{equation}

\end{lem}

\begin{proof}
 We refer to  \cite{Ku85} for a proof of \eqref{2.27} when $b\neq 2$,
as well as  \eqref{2.27-1}.
Here, we prove only  the case  $b=2$.

 First for $C=\max_{1\leq \rho \leq 2} [\rho
\mu]^{-1}$ we have
\begin{align*}
\int_B     |\phi|^2 /\rho  dm & \leq C \int_{1\leq \rho \leq 2} |\phi|^2 \mu dm +\int_{0\leq \rho \leq 1} |\phi|^2/\rho dm \\
& \leq C \int_B |\phi|^2 \mu dm +\int_0^1\frac{G^2}{\rho}d\rho,
\end{align*}
where we have used  the spherical coordinator representation with $\rho=2-r^2$ and
\begin{equation}\label{g2}
G^2(\rho)= - \int_{|\xi|=1} |\phi(r\xi)|^2 r^{N-1}dS_\xi \cdot \left(
\frac{d \rho }{d r} \right)^{-1}=\frac{1}{2} \int_{|\xi|=1} |\phi(r\xi)|^2 r^{N-2}dS_\xi.
\end{equation}
Note that from $\phi  \in \Hg$ one can verify that  $G(0)=0$.  It is known (see \cite{KuPe03}))  that
\begin{equation*}
\int_0^1 \left(\int_0^x g(t) dt\right)^2 \frac1x dx \leq C \int_0^1 g^2(x) x
|\ln x|^{2} dx.
\end{equation*}
Thus,
\begin{equation}\label{000}
\int_0^1 \frac{G^2}{\rho}d\rho \leq C \int_0^1 (G_\rho)^2 \rho |\ln \rho |^2 d\rho \leq  C \int_0^1 \frac{G^2_r}{r^2}\mu d\rho
\leq C \int_0^1 (G_r)^2 \mu d\rho,
\end{equation}
where we have used the fact that $\rho |\ln \rho|^2\leq \mu=\rho\ln^2\left(e/\rho\right)$.  Differentiation of (\ref{g2}) in term of $r$  leads to
\begin{align*}
2GG_r &=  \int_{|\xi|=1} \phi \nabla \phi \cdot \xi r^{N-2}dS_\xi + \frac{N-2}{2} \int_{|\xi|=1} |\phi(r\xi)|^2 r^{N-3} dS_\xi.
\end{align*}
Squaring both sides and  using the Cauchy-Schwartz inequality  we obtain
$$
 4G^2(G_r)^2 \leq 2 \int_{|\xi|=1}\phi^2 r^{N-2}dS_\xi \int_{|\xi|=1}|\nabla \phi|^2r^{N-2}dS_\xi +\frac{(N-2)^2}{2} \left( \int_{|\xi|=1}\phi^2 r^{N-2}dS_\xi  \right)^2.
 $$
 where we have used the fact $r\geq 1$.  Hence
 $$
(G_r)^2  \leq  \int_{|\xi|=1}|\nabla \phi(r\xi)|^2r^{N-2}dS_\xi +\frac{(N-2)^2}{2}G^2,
$$
which inserted into (\ref{000}) ensures that the term $\int_0^1
\frac{G^2}{\rho}d\rho$ is also bounded by $C\|\phi\|^2_{H^1_\mu}$.
The proof is now complete.
\end{proof}

\subsection*{Energy estimates} We return now to the bilinear operator  $\mathcal{B}$.
\begin{lem}[Energy estimates]\label{lem Energy} For any $t$, there exists a constant
$C$ which is dependent on $N,b$ such that
\begin{enumerate}
  \item[(1)] for  $w(t,\cdot) \in \Hg$
\begin{equation}\label{estimate B}
\frac14 \int |\nabla w|^2 \mu dm \leq \mathcal{B}[w,w;t] +
C(1+|\kappa|^2) \int w^2\mu dm;
\end{equation}
  \item[(2)] for $\psi(t,\cdot) \in H^1_\mu$  and $\phi \in \Hg$,
\begin{equation}\label{estimate B2}
 |\mathcal{B}[\psi,\phi;t]| \leq C(1+|\kappa|) ||\psi||_{H^1_\mu} ||\phi||_{H^1_\mu}.
\end{equation}
\end{enumerate}
\end{lem}

\begin{proof} From (\ref{bb}) it follows
\begin{eqnarray}\label{b1}
     \frac12 \int \nabla w \cdot \nabla \phi \mu dm & =&  \mathcal{B}[w,\phi;t] + \int \kappa m \cdot
  \nabla \phi w \mu dm  +\int  K  w \phi dm,
   \end{eqnarray}
   where $K$ is given in (\ref{K}).

(1) If  $0<b<2$,  then $K=0$; hence
    \begin{eqnarray}\label{bw}
     \frac12 \int |\nabla w|^2 \mu dm&=& \mathcal{B}[w,w;t] + \int \kappa m \cdot
 \nabla ww \mu dm\\ \notag
   &\leq& \mathcal{B}[w,w;t] + \frac14 \int |\nabla w|^2 \mu dm + b|\kappa|^2 \int
w^2 \mu dm
    \end{eqnarray}
and
\begin{eqnarray*}
|\mathcal{B}[\psi,\phi;t]| &\leq&  \frac12 \int |\nabla \psi||\nabla
\phi| \mu dm
+ \sqrt{b} |\kappa|  \int |\psi||\nabla \phi|\mu dm \\
&\leq& C(1+|\kappa|) ||\psi||_{H^1_\mu} ||\nabla \phi||_{L^2_\mu}.
\end{eqnarray*}
(2) For $b\geq 2$, it suffices to estimate the $K$-related term. If
$b=2$, we have
$$
K=(N+2\kappa m\cdot m)\ln \frac{e}{\rho}\leq (N+2b|\kappa|)\sqrt{\mu
\mu^*}.
$$
If $b>2$, we have
\begin{align*}
K&  =\left(\frac{b}{2}-1 \right)\rho^{1-b/2} (N+2\kappa m \cdot m)\\
  & \leq  \left(\frac{b}{2}-1\right)(N+2b|\kappa|)\sqrt{\mu \mu^*}.
\end{align*}
Hence for $b\geq 2$ we have
\begin{align*}
\int K w^2 dm & \leq C(1+|\kappa|)\int w^2 \sqrt{\mu \mu^*}dm \\
& \leq  \ep \int w^2 \mu^*dm + C_\ep(1+|\kappa|^2) \int w^2 \mu dm.
\end{align*}
This when added upon right side of  (\ref{bw}) using \eqref{2.27}
with some small $\ep$  leads to \eqref{estimate B}.  Using
\eqref{2.27} again we have
$$
\left| \int K \psi \phi dm \right|\leq C(1+|\kappa|) \int
|\psi||\phi| \sqrt{\mu\mu^*}dm \leq C(1+|\kappa|) ||\psi||_{H^1_\mu}
||\phi||_{H^1_\mu},
$$
which when combined with  the above  estimate for $b<2$ gives
\eqref{estimate B2}.
\end{proof}

\subsection*{A priori estimate}
\begin{lem}[A priori estimates]\label{lem10} Let $w$ be a weak solution to (\ref{ww}). Then
\begin{equation}\label{w_estimate4}
\sup_t ||w(t,\cdot)||^2_{L^2_\mu} +\frac12 ||w||^2_{L^2_tH^1_\mu} \leq e^{F(|\kappa|)}  \left( ||w_0||^2_{L^2_\mu} +
||h||^2_{L^2_t(\Hg)^*}\right).
\end{equation}
with $F$ defined in (\ref{F}),  and furthermore
\begin{equation}\label{w_estimate}
||w||^2_\mathcal{H} \leq e^{F(|\kappa|)} (||w_0||^2_{L^2_\mu} + ||
h||^2_{L^2_t((\Hg)^*)}).
\end{equation}
\end{lem}
\begin{proof}
From the weak solution definition in (\ref{4.9}) we have for any
$\phi \in \Hg$
 \begin{equation}\label{wp}
(\p_t w, \phi)_{\Hg} + \mathcal{B}[w,\phi;t] = ( h, \phi)_{\Hg}.
\end{equation}
By  \eqref{estimate B2},  $(\p_t w, \phi)_{\Hg} $ is bounded by
$$
\dis ||h||_{(\Hg)^*} ||\phi||_{H^1_\mu} + C(1+|\kappa|)
||w||_{H^1_\mu}||\phi||_{H^1_\mu}.
$$
Hence
\begin{equation}\label{w_estimate5}
||\partial_t w||_{(\Hg)^*} \leq ||h||_{(\Hg)^*} + C(1+|\kappa|)
||w||_{H^1_\mu}.
\end{equation}
Next we set $\phi=w$ in (\ref{wp}) and use  \eqref{estimate B} to
have
\begin{align*}
\frac12 \frac{d}{dt}\|w\|^2_{L^2_\mu}  + \frac 14 \int |\nabla
w|^2\mu dm
& \leq ||h||_{(\Hg)^*}\|w\|_{H^1_\mu}   + C(1+|\kappa|^2) \|w\|^2_{L^2_\mu} \\
&  \leq 2  ||h||^2_{(\Hg)^*}+ \frac{1}{8}\|w\|^2_{H^1_\mu}+
C(1+|\kappa|^2) \|w\|^2_{L^2_\mu}.
\end{align*}
Hence
\begin{equation}\label{00}
\frac{d}{dt}\|w\|^2_{L^2_\mu} + \frac 14 \|w\|^2_{H^1_\mu}\leq
C(1+|\kappa|^2) \|w\|^2_{L^2_\mu}+4 ||h||^2_{(\Hg)^*},
\end{equation}
and therefore by Gronwall's inequality,
\begin{equation*}
\sup_t ||w(t,\cdot)||^2_{L^2_\mu} +\frac12 ||w||^2_{L^2_tH^1_\mu} \leq e^{C(T
+\int_0^T |\kappa|^2dt)} \left( ||w_0||^2_{L^2_\mu} +
||h||^2_{L^2_t(\Hg)^*}\right),
\end{equation*}
which  together with \eqref{w_estimate5} yields \eqref{w_estimate}.
\end{proof}

\begin{proof}[Proof of Theorem \ref{thm6}]
We construct a weak solution to \eqref{ww} using the Galerkin
approximation. Let $\{\phi_i\}$ be a basis of $\Hg$ and $L^2_\mu$.
Then an approximate solution $w_l$ in a finite dimensional space is
defined as $\dis w_l= \sum_{i=1}^l d^l_i(t)\phi_i$. Here $d^l_i(t)$
is a unique solution to a system of linear differential equations,
\begin{eqnarray*}
(\p_t w_l, \phi_j)_{\Hg} + \mathcal{B}[w_l,\phi_j;t] =
(h,\phi_j)_{\Hg}, \\
d_i^l(0)=((\phi_i, \phi_j)_{L^2_\mu}^{-1}(w_0, \vec
\phi)_{L^2_\mu})_i,
\end{eqnarray*}
where $\vec \phi=(\phi_1, \cdots, \phi_l)^\top$.
Using the same argument as that in the proof
of Lemma \ref{lem10}, we obtain estimates for $w_l$ such that
\begin{eqnarray*}
||w_l||^2_{L^2_tH^1_\mu} + ||\p_t w_l||^2_{L^2_t(\Hg)^*} \leq e^{F(|\kappa|)}
\left( ||w_0||^2_{L^2_\mu} + ||h||^2_{L^2_t(\Hg)^*}\right).
\end{eqnarray*}
Extracting a subsequence and passing to the limit give a weak
solution $w$ in $\H$. The uniqueness follows from the a priori estimate
\eqref{w_estimate}.
\end{proof}

To return to the Fokker-Planck problem (\ref{ff})  we will also need
the following

\begin{lem}\label{hq} Let  $h=-\partial_t q -\mu^{-1}L[q]$. If  $q\in C_t^1H^1_\mu$  and $\kappa\in L^2_t$ with
$\text{Tr}(\kappa)=0$, then
\begin{equation}\label{hh}
\|h\|^2_{L^2_t(\Hg)^*} \leq C\int_0^T(1+|\kappa|^2) \|q(t)\|^2d\tau.
\end{equation}
\end{lem}
\begin{proof}  For $q\in C_t^1H^1_\mu $,  it is obvious that $\p_t q \in L^2_t(\Hg)^*$
since $H^1_\mu \subset (H^1_\mu)^* \subset (\Hg)^*$.  In order to show $\mu^{-1}L[q]
\in L^2_t(\Hg)^*$, we use integration by parts and
\eqref{estimate B2} to get
\begin{eqnarray*}
\left|\int \mu^{-1}L[q] \phi \mu dm \right| = \left|
\mathcal{B}[q,\phi;t]\right| \leq C(1+|\kappa|) ||q(t,\cdot)||_{H^1_\mu}
||\phi||_{H^1_\mu}, \; \forall \phi \in C^\infty_c.
\end{eqnarray*}
Since $C^\infty_c$ is a dense subset of $\Hg$, for any $\phi \in \Hg$ with  $||\phi||_{H^1_\mu}=1$, we have
\begin{equation}
|(\mu^{-1}L[q],\phi)_{\Hg}| \leq C(1+|\kappa|)||q(t,\cdot)||_{H^1_\mu}.
\end{equation}
Taking the $L^2$ norm in $t$ leads to the desired estimate.
\end{proof}

Theorem \ref{thm6} and Lemma \ref{hq}  lead to the following  result
for problem (\ref{ff}) with general Dirichlet boundary condition.
\begin{thm}\label{prop1}
Suppose that $f_0\nu^{-1}\in L^2_\mu$,   $q\in C_t^1H^1_\mu$ and
$\kappa \in L^2_t$ with $\text{Tr}(\kappa)=0$. Then for any $T>0$ the
Fokker-Planck problem (\ref{ff}) has a unique solution $f$ such that
\begin{equation}\label{fw}
f=\nu(w+q) \quad {\rm with}\;  w\in \Ht \;\; {\rm for}\;\; 0<t  \leq T.
\end{equation}
Moreover,  for $F$ defined in (\ref{F}),
\begin{equation}\label{wth13}
\sup_t ||w(t,\cdot)||^2_{L^2_\mu} +\frac12 ||w||^2_{L^2_tH^1_\mu}
\leq e^{F(|\kappa|)} \left( ||w_0||^2_{L^2_\mu} + \int_0^T
(1+|\kappa(t)|^2) ||q(t)||^2dt \right).
\end{equation}
\end{thm}
\begin{proof}
The estimate (\ref{wth13}) follows from (\ref{w_estimate4}) and the estimate in Lemma \ref{hq}, with $Fe^F$ replaced by $e^F$.

We now prove uniqueness of $f$  in terms of $q|_{\partial B}$.
Let $f_i (i=1, 2)$ be two solutions with $q_i$ such that
$q_1|_{\partial B}=q_2|_{\partial B}$ and same initial data $f_0$.
Set $w=(f_2-f_1)\nu^{-1}$, then $w$ solves w-problem (\ref{ww})
with $w_0\equiv h\equiv 0$. Hence $w\equiv 0$, leading to
$f_1=f_2$.
\end{proof}

\begin{rmk}
As mentioned in Section \ref{sec2} that $\Hg = H^1_\mu$ if $b\geq
6$,i.e., the trace of $q\in H^1_\mu$ vanishes if $b\geq 6$. Thus,
the boundary condition \eqref{LFPE3} is nothing but a zero Dirichlet
boundary condition. In Section \ref{section6}, we show the
well-posedness with a nonzero Dirichlet boundary condition for
$b\geq 6$ using  yet a different transformation.
\end{rmk}

\subsection{Probability density function} So far we have discussed well-posedness of the
initial-boundary value problem (\ref{ff}) for  $b>0$ and any given
$q$. We now turn to the question of which $q$ corresponds to the
probability density, i.e., non-negative solution with  constant mass for all time.

\begin{prop}\label{prop8}Let  $f(t,m)$ be the solution to  problem \eqref{ff}  obtained in Theorem \ref{prop1}.  If $f_0\geq
0$ and $q(t,m)|_{\p B} \geq 0$ almost everywhere,  then $f$ remains nonnegative for $t>0$.
\end{prop}
\begin{proof}
We adapt an idea  from  \cite{Ch09}. Let
$f^{\pm}$ be the positive and negative parts of the solution $f$
such that $f=f^+-f^-$.  Obviously, $w^\pm:=f^\pm\nu^{-1}\in H^1_\mu$
and $q|_{\p B}\geq 0$. This implies that the trace of $w^-$ at the
boundary vanishes, so
\begin{equation*}
w^-\in \Hg.
\end{equation*}
From  the equation
\begin{equation*}
\p_t w\mu + L[w]=0,
\end{equation*}
which is transformed from \eqref{LFPE1} it follows that
$$
(\partial_t w, w^-)_{\Hg}+B[w, w^-;t]=0.
$$
Since $(\partial_t w^+, w^-)_{\Hg}$ and $\int L[w^+]w^- dm $ vanish,
hence
\begin{equation*}
\frac12 \frac{d}{dt}  \left(\int |w^-|^2\mu dm \right) + \mathcal{B}[w^-,w^-;t] =0.
\end{equation*}
The coercivity of $\mathcal{B}$, \eqref{estimate B}, gives
\begin{equation*}
\frac12 \frac{d}{dt}   \left(\int |w^-|^2\mu dm \right) + \frac14 \int |\nabla w^-|^2 \mu
dm \leq C(1+|\kappa|^2) \int |w^-|^2 \mu dm.
\end{equation*}
Hence
\begin{equation*}
\sup_t ||w^-(t,\cdot)||^2_{L^2_\mu} \leq ||w^-_0||^2_{L^2_\mu} e^{F(|\kappa|)}
\end{equation*}
for $T>0$. Since $w^-_0=0$, $||w^-(t,\cdot)||^2_{L^2_\mu}=0$ for
all $0\leq t\leq T$.
\end{proof}

\begin{prop}\label{prop9} Let $f$ be a solution to the Fokker-Planck problem \eqref{ff} obtained
in Theorem \ref{prop1}. Suppose $b\geq 2$ and  $q(t,m)|_{\p B} \geq 0$.   If $q|_{\p B}=0$ for all $t\in [0,T]$, then
\begin{equation*}
\int\! f(t,\cdot) dm= \int \! f_0 dm, \quad t\in [0,T],
\end{equation*}
and vice versa.
\end{prop}
\begin{proof}
It suffices to prove the claim for smooth enough $f$ since  the
general case can be treated by an approximation as in \cite{LiSh08}.
We rewrite \eqref{LFPE1} as
\begin{equation*}
\p_tf = - \nabla \cdot(\kappa m f) + \nabla \cdot \left(\rho^{b/2}
\nabla\frac{f}{\rho^{b/2}}\right).
\end{equation*}
First, we take a test function $ \phi_\ep(m)=\phi_\varepsilon
(|m|)\in C^\infty_c(\R^N)$ converging to $\chi_B$ as $\varepsilon
\to0$ such that
\begin{equation*}
\phi_\varepsilon(|m|)=\left\{
       \begin{array}{ll}
         1, & |m| \leq \sqrt b- \varepsilon \\
         0, & |m| \geq \sqrt b-  \varepsilon/2
       \end{array}
     \right., \quad |\nabla \phi_\ep| \leq C\frac1\ep
\end{equation*}
and for any smooth $g$
\begin{equation}\label{cutoff}
\int_{\sqrt{b}-\epsilon}^{\sqrt{b}-\epsilon/2}  g(r) \phi'_\ep(r)dr
\to - g(\sqrt b) \quad \text{as~} \ep\to 0,
\end{equation}
where $\dis \phi'_\ep(r)=\nabla \phi_\ep \cdot \frac{m}{|m|}$.

One can construct such a $\phi_\ep$ by mollifiers, for example
\begin{equation*}
\phi_{\ep}(m)=\int_{B_{\sqrt{b}-3\ep/4}}\eta_{\ep/4}(m-m') dm'
\end{equation*}
where
\begin{equation*}
\eta_{\ep}(m)=\frac{1}{\ep^N}\eta(m/\ep),\quad \eta(m)=\left\{
                                                         \begin{array}{ll}
                                                           Ce^{-\frac{1}{1-|m|^2}}, & |m|<1 \\
                                                           0, &
|m|\geq1
                                                         \end{array}
                                                       \right.,
\end{equation*}
and $C$ is the normalizing constant.

Since $ \nabla \phi_\varepsilon$ is supported in $B^\varepsilon: =
B_{\sqrt b-\varepsilon/2}\setminus B_{\sqrt b-\varepsilon}$, hence
\begin{equation}\label{dtf}
\frac{d}{dt} \int_B f \phi_\varepsilon dm = \int _{B^\varepsilon}
f\kappa m \cdot \nabla \phi_\varepsilon dm - \int_{B^\ep} \rho^{b/2}
\nabla \left(\frac{f}{\rho^{b/2}}\right) \cdot \nabla \phi_\ep dm.
\end{equation}
By  $w=f\nu^{-1}$, the right hand side reduces to
\begin{equation}\label{5.3}
 \int_{B^\ep}(w\kappa m-  \nabla w)\cdot \nabla \phi_\ep \nu dm - \int_{B^\ep} w  \rho^{b/2}
\nabla \phi_\ep \cdot \nabla (\nu \rho^{-b/2}) dm.
\end{equation}
The first term converges to $0$. Indeed,
\begin{eqnarray*}
\left| \int_{B^\ep} (w\kappa m-  \nabla w) \cdot \nabla \phi_\ep \nu
dm \right| &\leq & \left(\int_{B^\ep} |w\kappa m-  \nabla w|^2 \mu
dm \right)^{1/2}\left(\int_{B^\ep} |\nabla
\phi_\ep|^2\frac{\nu^2}{\mu}dm\right)^{1/2}.
\end{eqnarray*}
Since ${\nu^2}/{\mu}=\rho^{b/2}$ for $b\geq 2$, by mean value theorem  there exists  $r \in
(\sqrt{b}-\varepsilon, \sqrt{b}-\varepsilon/2)$ such that
\begin{equation*}
\int_{B^\ep} |\nabla \phi_\ep|^2 \frac{\nu^2}{\mu} dm =
\frac{\varepsilon}{2}\int_{\partial B_r}|\nabla\phi_\ep|^2
\rho^{b/2}dS  \leq C \ep^{b/2-1},
\end{equation*}
which is uniformly bounded for $b\geq 2$.  Using $w\in H^1_\mu$, we obtain  $\int_{B^\ep} |w\kappa m-  \nabla w|^2 \mu dm \to 0$
as $\ep \to 0$.  Hence the first term  in (\ref{5.3}) converges to $0$.

On the other hand, for $\dis C_0=\left\{
           \begin{array}{ll}
             -2, & b=2 \\
             2-b, & b>2
           \end{array}
         \right.
$
\begin{eqnarray*}
-\int_{B^\ep} w  \rho^{b/2}  \nabla \phi_\ep \cdot \nabla (\nu
\rho^{-b/2}) dm &=& C_0 \int_{B^\ep} w \nabla \phi_\ep \cdot m dm \\
&=&  C_0 \int_{\sqrt{b}-\epsilon}^{\sqrt{b}-\epsilon/2}  \int_{\p B_r} w r \phi'_\ep(r) dS dr\\
&=& C_0 \int_{\sqrt{b}-\epsilon}^{\sqrt{b}-\epsilon/2}  \left( r
\int_{\p B_{r}} w dS \right)\phi'_\ep(r) dr.
\end{eqnarray*}
Due to \eqref{cutoff} this converges to
\begin{equation*}
-C_0 \sqrt b \int_{\p B} w dS = -C_0 \sqrt b \int_{\p B} q dS.
\end{equation*}
Since $C_0\neq 0$, this shows that $\dis \frac{d}{dt}\int_B f dm=0$
if and only if $\dis\int_{\p B} q dS =0 $, or $q|_{\p B}=0$.
\end{proof}
\begin{rmk}
In Proposition \ref{prop9}, the assumption $b\geq 2$ is sharp. In
the case $b<2$, we need to consider nontrivial $q\not=0$ since the
equilibrium profile $f_{eq}=\rho^{b/2}$ satisfies
\begin{equation*}
q\big|_{\p B}=\rho^{b/2} \nu^{-1}|_{\p B} = 1.
\end{equation*}
This requirement is also consistent with \cite{LiLi08}, in which it
was shown that when $b<2$, $f\nu^{-1}|_{\p B}=q|_{\p B}$ is
necessarily prescribed and each solution depends on the choice of
$q$. It would be interesting to figure out  a particular $q$  for
which the corresponding solution when $b<2$ is a probability
density.
\end{rmk}

\section{The Fokker-Planck equation}\label{section4}
In this section, we show the well-posedness of the FPE \eqref{FPE} including $x$
variable. The result is stated as follows.
\begin{thm}\label{thm3}
Suppose that for $b>0$ and any integer $s>N/2+1$, $\nabla \cdot
v=0$ and
\begin{equation}\label{regularity}
v \in C_tH_x^s \cap L^2_t H^{s+1}_x, \quad f_0\nu^{-1} \in H^s_x
L^2_\mu,\quad q\in C^1_t H^{s+1}_x H^1_\mu, \; 0<t<T
\end{equation}
for any $T>0$. Then \eqref{FPE} has a unique solution $f=\nu(w+q)$ satisfying
\begin{eqnarray} \label{estimate f}
\dis\sup_t |w|^2_{0,s} + \frac12 \int_0^T |\nabla_m w|^2_{0,s}dt
\leq e^{F(|v|_{s+1})} \left(|w_0|^2_{0,s} +  \|q\|_{1, 1, s+1}^2
\right),
\end{eqnarray}
where  $F$ was defined in (\ref{F}).
\end{thm}

The proof of Theorem \ref{thm3} consists of two parts:
first we show the existence of the solution $f$ to problem \eqref{FPE} by using the flow map, followed by proving
regularity in $x$ inductively such that $ w \in C_tH^s_xL^2_\mu \cap L^2_t H^s_x H^1_\mu$ with $v, f_0$
and $q$ given in (\ref{regularity}). 
In the second step, we derive estimate (\ref{estimate f}) directly from (\ref{FPE})
to control $f$ in terms of the given data. The uniqueness can be obtained from the estimation \eqref{estimate f}
as performed in the proof of Theorem \ref{prop1}.



First, we state a technical lemma.
\begin{lem}\label{lem17}
Suppose that $\psi\in H^1_\mu$ and $\phi \in \Hg$. Then for the
trace map $\dis \mathcal{T}: W^{{1,1}}(B)\to L^1(\p B)$
\begin{equation}\label{3.26}
\mathcal{T}(\psi\phi\mu)=0.
\end{equation}
\end{lem}
\begin{proof}
Since $C^\infty_c$ is a dense subset of $\Hg$, it suffices to show that
for a fixed $\psi \in H^1_\mu$ and any $\phi\in C^\infty_c$
\begin{equation}\label{note1}
||\psi\phi\mu||_{W^{1,1}} \leq C ||\phi||_{H^1_\mu}.
\end{equation}
Then, the standard trace
theorem asserts that $\mathcal{T}(\psi\phi\mu)$ is well-defined in $L^1(\p B)$ and it vanishes,    also  $\mathcal{T}$  is a continuous map with respect to $\phi$,
we can thus conclude \eqref{3.26} for any $\phi \in \Hg$ by passing to the limit of sequence $\phi_n \in C^\infty_c$ such that $\phi_n \to \phi$.

\eqref{note1} is indeed the case. It is obvious that
$\psi\phi\mu, \nm \psi \phi\mu$ and  $ \psi \nm\phi \mu$ are integrable.
For $b\neq 2$, $|\nm \mu| \leq C \sqrt{\mu\mu^*}$
and \eqref{2.27} yield
\begin{equation*}
\int |\psi \phi \nm \mu dm| \leq C
\|\psi\|_{L^2_\mu}||\phi||_{H^1_\mu}.
\end{equation*}
For $b=2$,
$$
\dis |\nm \mu| \leq C (\wet + \we) \leq C(\wet +
\sqrt{\mu\mu^*}) .
$$
Using \eqref{2.27-1} and $\psi \in H^1_\mu$, we obtain $\psi\in
L^2_{-1+\delta}$ for any $\delta>0$. Hence
\begin{eqnarray*}
\left|\int \psi \phi \wet dm \right| \leq C \left(\sqrt{\int \psi^2
\rho^{-1+\delta}dm} \sqrt{\int \phi^2 \rho^{1-\delta}\ln^4(\frac
e\rho)dm}\right).
\end{eqnarray*}
It follows that for any $b>0$
\begin{equation*}
\int |\psi\phi\mu| + |\nm(\psi\phi\mu)| dm < C ||\psi||_{H^1_\mu}||\phi||_{H^1_\mu}
\end{equation*}
as we desired.
\end{proof}

The main ingredient for the proof of Theorem  \ref{thm3} is to use the
calculus inequalities in the Sobolev spaces, see Appendix 3.5 of
\cite{MB2002}: for any positive integer $r>0$ and $u, v \in
L_x^\infty \cap H^r_x$,
\begin{eqnarray}
\label{sobolev calculus}
\sum_{|\gamma| \leq r} || \p^\gamma(uv)-u\p^\gamma v||_{L^2} &\leq& C
\left( ||\nabla u||_{L^\infty} ||v||_{H^{r-1}}+ ||u||_{H^r}||v||_{L^\infty}\right), \\
||uv||_{H^r}&\leq& C (||u||_{L^\infty}||v||_{H^r}+||u||_{H^r}||v||_{L^\infty}).
\label{sobolev calculus-2}
\end{eqnarray}
Note that (\ref{sobolev calculus}) remains
valid when $\partial^\gamma$ on the left hand is replaced by the
corresponding difference operator. \\

\begin{proof}[Proof of Theorem \ref{thm3}]~\\
\textbf{Step1 (well-posedness)} Let a particle path be defined by
\begin{equation*}
\p_t x(t,y)= v(t,x(t,y)),\quad x(0,y)=y,
\end{equation*}
along which the distribution function $\tilde f (t,y,m):=
f(t,x(t,y),m)$ solves
\begin{subequations}\label{tildef}
\begin{eqnarray}
\p_t\tilde f+ \mathcal{L}[ \tilde f]  &=&0,  \label{4.35}\\
\tilde f(0,y,m)&=&f_0(y,m),\\
\tilde f(t,y,m)\nu^{-1}|_{\p B}&=& \tilde q (t,y,m)|_{\p B}. \label{4.35-2}
\end{eqnarray}
\end{subequations}
Here $\mathcal{L}$ is defined in (\ref{of}) with $\kappa$ replaced
by $\tilde \kappa(t,y) = \nabla v(t,x(t,y))$, and $\tilde q
(t,y,m):= q(t,x(t,y),m)$.

In order to show existence of the solution to  \eqref{FPE} under the conditions $v\in C_tH^s_x
\cap L^2_t H^{s+1}_x$ and $\nabla \cdot v=0$,  it suffices to
show that \eqref{tildef} has a solution $\tilde f = \nu (\tilde w +
\tilde q)$ such that
\begin{equation*}
\tilde w: = w(t,x(t,y),m) \in C_tH^s_y L^2_\mu \cap L^2_t H^{s}_y
H^1_\mu,
\end{equation*}
assuming that
\begin{equation}\label{f_0}
\tilde \kappa \in  L^2_t H^s_y, \quad  w_0
\in H^s_y L^2_\mu,\quad \tilde q \in C^1_t H^{s}_y H^1_\mu.
\end{equation}
These follow from (\ref{regularity}) since $|\tilde
\kappa(t)|_{s}\leq C |v(t)|_{s+1}$ for $t>0$, $w_0(x, m)
=f_0\nu^{-1}-\tilde q(t=0)=w_0(y, m)$,  and $\|\tilde q\|_{1, 1, s}
\leq C\|q\|_{1, 1,s+1}$, for which we have used $\partial_t \tilde
q=\partial_t q+v \cdot \nabla q$.

Using Theorem \ref{prop1} for each $y$,  there exists a unique solution $\tilde f$  such that
$$
\tilde f=\nu(\tilde w+\tilde q)
$$
with $\tilde w$ satisfying (\ref{wth13}), i.e.,
\begin{align}\notag
\sup_t ||\tilde w(t,y,\cdot)||^2_{L^2_\mu} +\frac12 ||\tilde
w(\cdot,y,\cdot) ||^2_{L^2_tH^1_\mu} & \leq e^{F(|\tilde
\kappa(\cdot, y)|)}  \left( ||w_0(y,\cdot)||^2_{L^2_\mu} \right.
\\ \label{+w}
& \qquad \left. +
\int_0^T(1+|\tilde \kappa(\cdot, y)|^2) ||\tilde q(t, y, \cdot)||^2dt\right).
\end{align}
Integration of (\ref{+w}) with respect to $y$, upon exchanging the
order of integration in $y$ and $m$, and using the Sobolev
inequality, $\sup_y |\tilde \kappa|\leq C |\tilde \kappa |_{s-1}$,
gives
\begin{align}\label{bb+}
\sup_t |\tilde w|^2_{0,0} +\frac12\int_0^T |\tilde w|^2_{1,0} dt  &
\leq e^{F( |\tilde \kappa|_{s-1})}\left( |w_0|^2_{0,0} + \|\tilde q\|^2_{1, 1,
0}\right).
\end{align}
Hence $\tilde w \in C_tL^2_yL^2_\mu \cap L^2_tL^2_yH^1_\mu$. On the
other hand, the right hand side of \eqref{+w} is uniformly bounded
in $y$, taking $\sup_y$ of (\ref{+w}) gives
\begin{equation}\label{supw}
\sup_{t,y}\|\tilde w(t, y, \cdot)\|^2_{L^2_\mu}
\leq e^{F(|\tilde \kappa|_{s-1})}(|w_0|^2_{0,s-1}+ ||\tilde
q||^2_{1,1,s-1}).
\end{equation}
We now use an induction argument to prove that $\tilde w \in
C_tH^{r}_yL^2_\mu \cap L^2_tH^{r}_yH^1_\mu$ for $0\leq r \leq s$,
and
\begin{equation}\label{induction}
\sup_t |\tilde w|^2_{0,r}+\frac12\int_0^T|\tilde w|^2_{1,r} dt \leq
e^{F(|\tilde \kappa|_{s})}(|w_0|^2_{0,s}+||\tilde q||^2_{1,1,s}).
\end{equation}
The case $r=0$ has been proved as shown in (\ref{bb+}). Suppose
(\ref{induction}) holds for $r=k$, we only need to show
(\ref{induction}) for $r=k+1\leq s$.

To prove regularity of $\tilde f$ in the $y$ variable, we use
difference quotients. Define the difference operator in the $y$
variable as
\begin{equation*}
\delta^\gamma:=\delta_1^{\gamma_1}\cdots \delta_N^{\gamma_N},\quad
\delta_i u(y):= \frac{1}{\eta}[u(y+\eta e_i)-u(y)].
\end{equation*}
Apply $\delta^\gamma$ to (\ref{tildef}) with $|\gamma|\leq s$,  then
\begin{subequations}\label{fD}
\begin{eqnarray}
\p_t\dg\tilde f + \mathcal{L}[\dg \tilde f] &=&  \nm\cdot J,\label{fD1}\\
\dg \tilde f(0,y,m)&=&\dg f_{0}(y,m),\\
\dg \tilde f(t,y,m)\nu^{-1}|_{\p B}&=& \dg \tilde q(t,y,m) |_{\p B},
\label{fD2}
\end{eqnarray}
\end{subequations}
where
\begin{equation}\label{J}
J=\tilde \kappa m \dg \tilde f -\dg(\tilde \kappa m \tilde f).
\end{equation}
This when transformed into the w-problem of form \eqref{ww} involves
the following non-homogeneous term
\begin{equation}\label{h+}
h=-\p_t \dg \tilde q -\mu^{-1}L[\dg \tilde q] + \nm \cdot J
\nu^{-1}.
\end{equation}
Using Theorem \ref{prop1} again for each $y$, $\dg \tilde f$ is the
unique solution to \eqref{fD} as long as $h \in L^2_t(\Hg)^*$.
Moreover,
$$
\dg \tilde f=\nu (\dg \tilde w+\dg \tilde q),
$$
where $\dg \tilde w$,  using \eqref{w_estimate4}, satisfies
\begin{eqnarray*} \sup_t ||\dg \tilde w(t,y,\cdot) ||^2_{L^2_\mu}
+\frac12 ||\dg \tilde w(\cdot,y,\cdot)||^2_{L^2_tH^1_\mu} \leq
e^{F(|\tilde \kappa(\cdot, y)|)}  \left( ||\dg w_{0} ||^2_{L^2_\mu}
+ ||h||^2_{L^2_t(\Hg)^*}\right).
\end{eqnarray*}
Integration in $y$ gives
\begin{align}\notag
 \sup_t |\dg \tilde w |^2_{0, 0} +\frac12 \int_0^T|\dg \tilde w|^2_{1,
 0}dt
& \leq e^{F(\sup_y |\tilde \kappa(\cdot, y)|)}  \left( |\dg
w_{0} |^2_{0, 0} + ||h||^2_{L^2_tL^2_y (\Hg)^*}\right) \\
\label{fdd} & \leq e^{F(|\tilde \kappa|_{s-1})}
\left( | w_0 |^2_{0, s} +||h||^2_{L^2_t L^2_y (\Hg)^*}\right).
\end{align}
We now turn to bound the last term in the above inequality. For any
$\phi \in \Hg$ and $J$ defined in (\ref{J}), Lemma \ref{lem17}
allows the use of integration by parts. Hence,
\begin{align*}
\left|\int \nabla_m \cdot J  {\nu^{-1}}\phi \mu dm \right|
& \leq
\left(\int |J \nu^{-1}| |\nu\nabla_m \frac{\mu}{\nu}||\phi| dm +
\int |J \nu^{-1}| |\nabla_m \phi| \mu dm \right) \\
& \leq C \|J\nu^{-1}\|_{L^2_\mu} (\|\phi\|_{L^2_{\mu^*}}+\|\nabla_m
\phi\|_{L^2_\mu}) \\
& \leq C ||J \nu^{-1}||_{L^2_\mu} ||\phi||_{H^1_{\mu}}.
\end{align*}
Here we have used $\dis |\nu \nabla_m \frac{\mu}{\nu}| \leq C
\sqrt{\mu^*\mu}$ and the embedding theorem (\ref{2.27}).
This together with Lemma \ref{hq} and (\ref{h+})
yields
\begin{equation}\label{4.9-1}
\|h\|^2_{L^2_tL^2_y (\Hg)^*} \leq C \int_0^T(1+\sup_y |\tilde
\kappa(t, y)|^2) \int \|\dg \tilde q(t, y, \cdot) \|^2dy dt +
C \int_0^T |J\nu^{-1}|^2_{0, 0}dt.
\end{equation}
For $|\gamma|\leq s$, the first term on the right side is bounded by
\begin{equation}\label{t1}
F(|\tilde \kappa |_{s-1})\|\delta^\gamma \tilde q\|^2_{1, 1, 0}\leq
F(|\tilde \kappa |_{s-1})\|\tilde  q\|^2_{1, 1, s}.
\end{equation}
To obtain (\ref{induction}) for $r=k+1\leq s$, it remains to
estimate the last term in (\ref{4.9-1}) with $|\gamma|=k+1$. In
fact,
\begin{align*}
|J\nu^{-1}|^2_{0, 0}& =
|(\dg(\tilde \kappa m \tilde f) -\tilde \kappa m \dg
\tilde f )\nu^{-1}|^2_{0, 0} \\
& \leq C (\sup_y |\nabla_y \tilde \kappa|^2 |\tilde f\nu^{-1}|^2_{0, k}
+|\tilde \kappa|^2_{k+1}\sup_y\|\tilde f\nu^{-1}\|_{L^2_\mu}^2)\\
& \leq C |\tilde \kappa |^2_{s}(|\tilde w |^2_{0, k}+ \sup_y \|\tilde w\|_{L^2_\mu}^2
+\|\tilde q\|^2_{1, 1, s}),
\end{align*}
where we have used (\ref{sobolev calculus}) with $\partial^\gamma$
replaced by $\dg$.

\noindent Using (\ref{induction}) for $r=k$ and (\ref{supw}) we have
$$
\int_0^T|J\nu^{-1}|^2_{0, 0}dt \leq e^{F(|\tilde \kappa|_{s})}(|w_0|^2_{0,s}+||\tilde
q||^2_{1,1,s}).
$$
This and (\ref{t1}) when inserted into (\ref{4.9-1}) gives a bound
for $\|h\|^2_{L^2_tL^2_y (\Hg)^*}$.  That bound combined with  (\ref{fdd})
yields
\begin{align*}
 \sup_t |\dg \tilde w |^2_{0, 0} +\frac12 \int_0^T |\dg \tilde w|^2_{1, 0}dt
  \leq e^{F(|\tilde \kappa|_{s})}  \left( |w_0 |^2_{0, s} +
 \|\tilde q\|^2_{1, 1, s} \right)<\infty, \quad |\gamma|=k+1.
  \end{align*}
Sending $\eta \to 0$ we obtain (\ref{induction}) with
$r=k+1$.  Hence, (\ref{induction}) holds for any $r\leq s$,
and thus the solution $f$ to \eqref{FPE} exists, and
\begin{equation*}
\sup_t| w|^2_{0,s} + \frac12\int_0^T| w|^2_{1,s} dt <\infty.
\end{equation*}
One may obtain an upper bound from (\ref{induction}) with
$r=s$ using the inverse map of $x=x(t, y)$. Nevertheless, the next step gives the claimed bound in
(\ref{estimate f}). \\
\textbf{Step2 (a priori estimate)} For a priori estimate, we consider the w-problem \eqref{W}
\begin{eqnarray}\label{step2-1}
\mu (\p_t    + v\cdot\nabla)w + L[w] &=& -\mu( \p_t   + v\cdot\nabla)q - L[q].
\end{eqnarray}
Recall that
\begin{equation*}
L[w]=-\frac12 \nabla_m\cdot (\nabla_m w \mu)+ \nabla_m \cdot (\kappa m w \mu)-K w.
\end{equation*}
Take $\gamma$ derivative in $x$-variable. Then, the left and right hand side of \eqref{step2-1} will be
\begin{eqnarray}\label{3.27}
I &=& \mu(\p_t + v\cdot \nabla) \p^\gamma w -\frac12 \nm \cdot(\nm \p^\gamma w \mu)\\
 &+& \mu[\p^\gamma((v\cdot\nabla)w)-(v\cdot \nabla) \p^\gamma w]\label{3.27-1}\\
 &+& \nm\cdot(\p^\gamma (\kappa m w \mu))\label{3.27-2}\\
 &-& \p^\gamma (K w),\\\label{3.27-3}
\label{3.28}
II &=& -\mu \p_t \p^\gamma q  +\frac12 \nm \cdot(\nm \p^\gamma q \mu)\\
 &-& \mu \p^\gamma((v\cdot\nabla)q)\label{3.28-1}\\
 &-& \nm\cdot(\p^\gamma (\kappa m q \mu))\label{3.28-2}\\
 &+& \p^\gamma (K q).\label{3.28-3}
\end{eqnarray}
We now estimate term by term of
\begin{equation}\label{lr}
\sum_{|\gamma|\leq s}\int\!\int\! I \partial^\gamma w dm dx= \sum_{|\gamma|\leq s}\int\!\int\! II \partial^\gamma w dm dx.
\end{equation}
Since $v$ is divergence free, the first two terms on the left hand side will be
\begin{equation*}
\frac12 \frac{d}{dt} |w|^2_{0,s}+\frac12|\nm w|^2_{0,s}.
\end{equation*}
Indeed, Cauchy inequality shows that the term related to  \eqref{3.27-1} is bounded by
\begin{eqnarray*}
\ep  |w|^2_{0,s} + C_\ep \sum_{|\gamma|\leq s}\int\int|\p^\gamma((v\cdot\nabla)w)-(v\cdot \nabla) \p^\gamma w|^2 \mu dm dx
\end{eqnarray*}
Now, we exchange the order of integration in $x$ and $m$,  and apply \eqref{sobolev calculus} to obtain
\begin{eqnarray*}
&~&  \ep  |w|^2_{0,s} + C_\ep \int \left(||\nabla v||^2_{L^\infty_x}||\nabla w(\cdot,m)||^2_{H^{s-1}_x} + ||v||^2_{H^s_x}||\nabla w(\cdot,m)||^2_{L^\infty_x}\right) \mu dm\\
&\leq& \ep  |w|^2_{0,s} + C_\ep  | v|^2_s|w|^2_{0,s},
\end{eqnarray*}
where the Sobolev inequality,  $|u|_{0}\leq C |u|_{s-1}$ for any
$u\in H^{s-1}_x$, is invoked in the last inequality. Similarly, the
term with \eqref{3.27-2} will be estimated as follows due to
\eqref{sobolev calculus-2};
\begin{eqnarray*}
   &~&\ep |\nm w|^2_{0,s} + C_\ep \sum_{|\gamma|\leq s} \int\int|\p^\gamma(\kappa m w)|^2 \mu dm dx\\
   &\leq& \ep |\nm w|^2_{0,s} + C_\ep \int \left(|\kappa|^2_{L^\infty_x}|w(\cdot,m)|^2_{s} + |\kappa|^2_{s}
   |w(\cdot,m)|^2_{L_x^\infty}\right) \mu dm\\
   &\leq& \ep |\nm w|^2_{0,s} + C_\ep |v|^2_{s+1}|w|^2_{0,s}
\end{eqnarray*}
Recall that
\begin{equation*}
K=\left\{
      \begin{array}{ll}
        0, & 0<b<2, \\
        (N+2\kappa m\cdot m)  \dis \ln \frac{e}{\rho}, &  b=2, \\
        (N+2\kappa m\cdot m) (b/2-1) \rho^{1-b/2}, & b>2.
      \end{array}
    \right.
\end{equation*}
Thus, we can express $K$ as
\begin{equation}\label{3.31}
K = c_1 \sqrt{\mu\mu^*} + c_2\kappa m\cdot m \sqrt{\mu\mu^*}
\end{equation}
for some positive constat $c_i$ depending on $N$ and $b$. We now
estimate the last term on the left hand side, by using
\begin{eqnarray*}
 \p^\gamma(K w)\p^\gamma w  = c_1
|\p^\gamma w|^2 \sqrt{\mu\mu^*} + c_2 \p^\gamma(\kappa m\cdot m w)
\p^\gamma w \sqrt{\mu\mu^*}.
\end{eqnarray*}
The Cauchy inequality and the embedding theorem \eqref{2.27} give
\begin{align*}
c_1 \sum_{|\gamma|\leq s}\int \!\! \int |\p^\gamma w|^2
\sqrt{\mu\mu^*}dm dx & =c_1 \int |w(t,\cdot, m)|_{s}^2\sqrt{\mu
\mu^*}dm \\
& \leq \epsilon \int |w(t, \cdot, m)|^2_{s}\mu^*dm +C_\epsilon \int
|w(t, \cdot, m)|^2_{s}\mu dm \\
& \leq \ep |\nm w|^2_{0,s} + C_\ep|w|^2_{0,s}.
\end{align*}
Similarly,
\begin{eqnarray*}
c_2 \sum_{|\gamma|\leq s} \int\int |\p^\gamma(\kappa m \cdot m
w)\p^\gamma w|\sqrt{\mu\mu^*} dm dx \leq \ep |\nabla_m w|^2_{0,s} +
C_\ep \int\!\! \int |\p^\gamma(\kappa m \cdot m w)|^2\mu dm dx.
\end{eqnarray*}
The last term, using \eqref{sobolev calculus-2} and the Sobolev
inequality for $\kappa=\nabla v$, is then bounded by
$$
C_\epsilon |v|^2_{s+1}|w|^2_{0, s}.
$$
Hence,
\begin{eqnarray*}
\left|\sum_{|\gamma|\leq s} \int\!\! \int \p^\gamma(K w)\p^\gamma w
dm dx \right| \leq \ep |\nm w|^2_{0,s} +
C_\ep(|v|^2_{s+1}+1)|w|^2_{0,s}.
\end{eqnarray*}
Now we turn to the right hand side, related to
\eqref{3.28}-\eqref{3.28-3}. The estimation is similar to that for
the left hand side. Except that here we have to assume higher
regularity of $q$ in $x$ than that of $w$ since  $\dis  \int
v\cdot\int \nabla \p^\gamma q \p^\gamma w \mu dm dx$ does not vanish
as $\dis \int v\cdot \int \nabla \p^\gamma w \p^\gamma w \mu dm dx$.
Indeed, the first two terms, related to \eqref{3.28} are bounded by
\begin{eqnarray*}
\ep |w|^2_{0,s} + C_\ep |\p_t q|^2_{0,s} + \ep|\nm w|^2_{0,s} + C_\ep |\nm q|^2_{0,s},
\end{eqnarray*}
and the other terms are estimated as follows;
\begin{eqnarray*}
\sum_{|\gamma|\leq s} \left|\int\!\! \int \p^\gamma(v\cdot \nabla
q)\p^\gamma w \mu dm dx \right|
&\leq & \ep |w|^2_{0,s} + C_\ep |v|^2_s|q|^2_{0,s+1}, \\
\sum_{|\gamma|\leq s} \left|\int\!\! \int \p^\gamma (\kappa m q) \nm
\p^\gamma w \mu dm dx \right|
&\leq & \ep |\nm w|^2_{0,s} + C_\ep |v|^2_{s+1}|q|^2_{0,s},\\
\sum_{|\gamma|\leq s} \left| \int\!\!\int \p^\gamma (K q) \p^\gamma
w dm dx \right| & \leq & \ep |\nabla_m w|^2_{0,s} + C_\ep
|q|^2_{0,s} + C_\ep |v|^2_{s+1}|q|^2_{0,s}.
\end{eqnarray*}
We combine all estimates for sufficiently small $\ep$ to obtain
\begin{equation}\label{5.2}
\p_t |w|^2_{0,s} +\frac12 |\nm w|^2_{0,s} \leq C
(|v|^2_{s+1}+1)\left( |w|^2_{0,s}+ (|q|^2_{1,s+1} + |\p_t
q|^2_{1,s+1})\right).
\end{equation}
We deduce that
$$
|w|^2_{0,s} +\frac12\int_0^t |\nm w|^2_{0,s} dt \leq  e^{F(|v|_{s+1})} \left(|w_0|^2_{0,s} + F(|v|_{s+1}) \|q\|^2_{1, 1, s+1} \right).
$$
Replacing $Fe^F$ by $e^F$ leads to  (\ref{estimate f}).
\end{proof}

\section{Coupled system}\label{section coupled}
In this section, we prove Theorem \ref{thm1} by the fixed point
argument as described in Section \ref{sec2}.

We begin with a key lemma, which will be used to estimate the stress $\tau$.
\begin{lem}\label{lem2}
Suppose that $\phi \in \Hg$. For any $\ep>0$ there exists $C_\ep$
such that
\begin{equation}\label{4.8}
\left|\int \phi\nu\rho^{-1} dm \right|^2 \leq C_\ep \int |\phi|^2\mu dm+ \ep
\int |\nm \phi|^2 \mu dm.
\end{equation}
\end{lem}

\begin{proof}
For $b>2$,  the Cauchy-Schwartz inequality yields
\begin{equation*}
\left|\int \phi dm \right |^2 \leq \int|\phi|^2 \mu dm \int \mu^{-1}dm.
\end{equation*}
For any $\ep>0$, taking  $C_\epsilon=\dis \int \mu^{-1}dm<\infty $,  we obtain
\eqref{4.8} for $b>2$.

For  $b\leq2$,  we define for fixed $M$,
\begin{equation*}
G=\{\phi\in \Hg: \int\phi \nu\rho^{-1}dm =1, ||\phi||_{H^1_\mu}\leq
M\}.
\end{equation*}
It suffices to prove
\begin{equation*}
l:=\inf_{\phi\in G} \int |\phi|^2 \mu dm >0.
\end{equation*}
Let $\{\phi_n\} \subset G$ be a sequence such that
\begin{equation*}
\lim_{n\to\infty} \int |\phi_n|^2 \mu dm = \inf_{\phi\in G} \int
|\phi|^2 \mu dm.
\end{equation*}
Since $\{\phi_n\}$ is bounded in $H^1_\mu$, by embedding
theorem \eqref{2.27}, there exists a subsequence $\{\phi_{n_k}\}$
such that
\begin{eqnarray*}
\phi_{n_k} \rightharpoonup \phi^* &\quad&\text{in~} H^1_{\mu},\\
\phi_{n_k} \rightharpoonup \phi^* &\quad&\text{in~} L^2_{\mu},\\
\phi_{n_k} \rightharpoonup \phi^* &\quad&\text{in~} L^2_{\mu^*}.
\end{eqnarray*}
Furthermore, since $\dis \sqrt{\frac{\mu}{\mu^*}} \in L^2_{\mu^*}$
for $b\leq 2$
\begin{eqnarray*}
\int \phi^* \nu\rho^{-1} dm &=& \int \phi^* \sqrt{\frac{\mu}{\mu^*}}\mu^* dm\\
&=& \lim_{n_k\to\infty }\int \phi_{n_k}
\sqrt{\frac{\mu}{\mu^*}}\mu^* dm =1.
\end{eqnarray*}
This shows that $\phi^* \in G$.  On the other hand,
\begin{equation*}
\int |\phi^*|^2\mu dm \leq \lim_{n_k\to\infty} \int
|\phi_{n_k}|^2\mu dm = l.
\end{equation*}
If $l=0$, then $\phi^*=0$ which is a contradiction to  $\phi^*\in
G$.
\end{proof}

The zero trace of $\phi$ is essential for the estimate (\ref{4.8}).
For the general case, i.e.,  for $\phi \in H^1_\mu$, one can only
have a weaker estimate.
\begin{lem}\label{lem2+}
If $\phi \in H^1_\mu$, then there exists $C$
such that
\begin{equation}\label{4.8+}
\left|\int \phi\nu\rho^{-1} dm \right|^2 \leq C \|\phi\|^2_{H^1_\mu}.
\end{equation}
\end{lem}
\begin{proof} For $b>2$, we have
\begin{equation*}
\left|\int \phi \nu\rho^{-1}dm \right|^2 \leq C \int |\phi|^2\mu dm, \quad C:= \int \mu^{-1} dm<\infty.
\end{equation*}
For $b\leq 2$,
\begin{eqnarray*}
\left|\int \phi \nu\rho^{-1}dm \right|^2 \leq C_\delta \int
|\phi|^2\rho^{-1+\delta}dm, \quad C_\delta:= \left(\int
\nu^2\rho^{-1-\delta}dm\right).
\end{eqnarray*}
We choose $\delta >0$ small enough so that $C_\delta$ is bounded.  On the other hand,  by  \eqref{2.27-1}  in Lemma \ref{lem1}   we have
\begin{eqnarray*}
\int |\phi|^2\rho^{-1+\delta}dm&\leq& C\int (|\phi|^2+|\nm \phi|^2)\rho^{b/2} dm = C\int (|\phi|^2+|\nm \phi|^2)\mu dm, \quad b<2\\
\int |\phi|^2\rho^{-1+\delta}dm &\leq& C\int (|\phi|^2+|\nm \phi|^2)\rho dm
\leq C \int (|\phi|^2+|\nm \phi|^2)\mu dm, \quad b=2.
\end{eqnarray*}
This completes the proof.
\end{proof}

We  now turn to the map
\begin{eqnarray*}
\F : &\mathbf{M} & \to \mathbf{M}\\
     &(u,\varpi)&  \mapsto (v,w),
\end{eqnarray*}
and
\begin{equation*}
\mathbf{M}= \left\{(v,w) :  \sup_{0\leq t\leq T} |v|^2_s \leq A_1,
\sup_{0\leq t\leq T} |w|^2_{0,s} + \frac12\int_0^{T} |\nm w|^2_{0,s}dt
\leq A_2 \right\}.
\end{equation*}
We first prove that,  given $v_0\in H^s_x, f_0\nu^{-1}\in H^s_xL^2_\mu$ and $q\in
C^1_tH^{s+1}_xH^1_\mu$, the map $\mathcal F$ is well defined, i.e.,  $\mathcal F
(\mathbf{M}) \subset \mathbf{M}$ for some $A_1, A_2, T$.

Let $(u, \varpi)\in \mathbf{M}$. It is now well known that  \eqref{NSE} has a unique solution $v$
such that
\begin{equation}\label{ineq3}
\sup_t |v|^2_s + \int_0^T |v|^2_{s+1} dt \leq |v_0|^2_s + C\int_0^T
|u|_s|v|^2_s dt+ \int_0^T |\tau|^2_sdt, \quad  s>N/2+1.
\end{equation}
By Gronwall's inequality  and  $\dis\sup_{0\leq t\leq T} |u|^2_s \leq A_1$, we have
\begin{equation}\label{4.6-2} \sup_t
|v|^2_s + \int_0^{T}|v|^2_{s+1} dt \leq \left(|v_0|^2_s +
\int_0^{T}|\tau|^2_sdt \right)e^{C\sqrt{A_1}T}.
\end{equation}
We proceed to estimate the stress term
$$
\int_0^T |\tau|_s^2dt=\int_0^T\sum_{|\gamma|\leq s}\int |\partial^\gamma \tau|^2dx dt,
$$
where using Lemma \ref{lem2},
\begin{align*}
 |\partial^\gamma \tau|^2 &= b^2 \left | \int_B m\otimes m\partial^\gamma(\varpi +q)\nu \rho^{-1} dm \right|^2\\
&\leq C_\ep \int |\partial^\gamma \varpi |^2\mu dm + \frac\ep2\int |\partial^\gamma \nm \varpi |^2\mu dm
+ 2b^4 \left| \int \partial^\gamma q \nu\rho^{-1}dm \right|^2.
\end{align*}
Using (\ref{4.8+}) the last term is uniformly bounded by
$$
C\|\partial^\gamma q(t, x, \cdot)\|^2_{H^1_\mu}\leq C\|q\|^2_{1, 1, s+1}.
$$
Hence for $(u, \varpi)\in \mathbf{M}$  we obtain
\begin{eqnarray}\label{tau}
\int_0^{T} |\tau|^2_s dt \leq  C_\ep T A_2  + \epsilon A_2 +  C T |q|_{1, 1, s+1}^2 \leq CT(A_2+\|q\|_{1, 1, s+1}^2)+\epsilon A_2,
\end{eqnarray}
where we have used the assumption  $q\in C^1_tH^{s+1}_xH^1_\mu$.

We choose $A_1$ as
\begin{equation}
A_1 = 2|v_0|^2_s e,
\end{equation}
 $A_2$ as
\begin{equation}\label{4.3}
A_2=(|w_0|^2_{0,s}+ \|q\|^2_{1, 1, s+1})e^{C(T+A_1)}
\end{equation}
for $ T\leq 1/(C\sqrt{A_1})$.

Hence, if  $T$ and $\ep$  are chosen small enough so that
\begin{equation*}
CT(A_2+ \|q\|_{1, 1, s+1}^2)+ \ep A_2 \leq \frac1{2e} A_1,
\end{equation*}
we get
\begin{equation}
e^{C\sqrt{A_1}T}  \left( |v_0|^2_s + CT(A_2+|q|_{1, 1, s+1}^2)+\epsilon A_2 \right)  \leq e(|v_0|_s^2 + \frac1{2e} A_1) \leq A_1.
\end{equation}
This together with  \eqref{4.6-2}, (\ref{tau})  gives
\begin{equation}\label{4.7}
\sup_t |v|^2_s + \int_0^{T_1}|v|^2_{s+1} dt\leq  A_1.
\end{equation}
Estimate \eqref{estimate f} in Theorem \ref{thm3}, \eqref{4.3} and \eqref{4.7} yield
\begin{equation}
\sup_t |w|^2_{0,s} + \frac12\int_0^{T} |\nm w|^2_{0,s} dt\leq A_2.
\end{equation}
So the map $\mathcal{F}$ is well defined in
$\mathbf{M}$.

Next, we show that $\mathcal{F}$ is a contraction mapping for small
enough $T$ using a weak norm on $\mathbf{M}$, i.e.
\begin{equation}
||(v,w)||^2_\mathbf{M}:=\sup_t |v|^2_0 +  \sup_t |w|^2_{0,0} +
\frac12 \int_0^{T} |\nm w|^2_{0,0} dt.
\end{equation}
Suppose that $v_i(i=1,2)$ are  solutions of the NSE \eqref{NSE} with
$u_i (i=1,2)$ and $\tau_i (i=1,2)$ computed from $\varpi_i (i=1,2)$
respectively.  Then we obtain
\begin{equation}\label{nv.1}
\p_t v + (u_2\cdot \nabla) v +(u \cdot \nabla)v_1 + \nabla p= \nabla \cdot \tau +
\Delta v, \quad v(0, \cdot)=0,
\end{equation}
where $v=v_2-v_1, u=u_2-u_1, p=p_2-p_1$, $\tau=\tau_2-\tau_1$ and
$\varpi=\varpi_2-\varpi_1$.
 Multiplication by $v$ to \eqref{nv.1}
and integration with respect to $x$ yield
\begin{eqnarray*}
\frac12 \frac d{dt}|v|^2_0 +\int (u\cdot \nabla v_1) v dx = - \int \tau
\nabla v dx - \int |\nabla v|^2 dx.
\end{eqnarray*}
Hence
\begin{eqnarray}
\frac{d}{dt} |v|_0^2 + |\nabla v|^2_0 &\leq& |u|^2_0 +|\tau|^2_0+
\sup_x|\nabla v_1|^2 |v|^2_0\nonumber\\
&\leq& |u|^2_0 +|\tau|^2_0+ A_1|v|^2_0.
\label{4.18}
\end{eqnarray}
Let $f_i$ be the
solutions to (\ref{FPE}) associated with $v_i(i=1,2)$.
Then
$$
w=(f_2-f_1)\nu^{-1}=:w_2-  w_1
$$
solves
\begin{subequations}
\label{contraction_f}
\begin{eqnarray}
&\dis\p_t w \mu + v_2 \cdot \nabla w \mu + L_2[w] = -v\cdot \nabla
w_1\mu -\nm  \cdot (\nabla v
m \tilde w_1 \nu )\frac{\mu}{\nu}, \label{4.15} &\\
&w(0,x,m)=0,&\\
&w(t,x,m)|_{\p B}=0,&
\end{eqnarray}
\end{subequations}
where $L_2[w]=L[w]$ defined in (\ref{lw}) with
$\kappa =\nabla v_2$. Note that $ w_i|_{\p B}= q|_{\p B}$,
i.e. $ w_i(t,x,\cdot) \in H^1_\mu$, so
$w(t,x,\cdot)\in \Hg$.

We deduce from \eqref{4.15} that
\begin{eqnarray*}
\frac12 \frac{d}{dt} |w|^2_{0,0} + \frac12 |\nm w|^2_{0,0} &\leq& \int\int
|\nabla v_2 m \cdot  \nm w \nu w | dmdx + \int\int
|Kw^2|dmdx\\
 &+& \int |v\cdot \int \nabla  w_1 w\mu dm | dx+
\int\left |\int \nm \cdot (\nabla v m w_1\nu)\frac{\mu}{\nu}wdm \right|dx.
\end{eqnarray*}
Similar to that led to  \eqref{estimate f},  first two terms on the right hand side are bounded by
\begin{equation*}
C_\ep(|v_2|^2_{s}+1)|w|^2_{0,0} + \ep|\nm w|^2_{0,0},
\end{equation*}
and the third term
\begin{eqnarray*}
\int |v\cdot \int \nabla  w_1 w\mu dm | dx &\leq& C\int |v|^2 \int |\nabla  w_1|^2\mu dm dx + \int\int |w|^2\mu dm dx \\
&\leq& C|v|^2_0|w_1|^2_{0,s} + |w|^2_{0,0}.
\end{eqnarray*}
The last term, using integration by parts with vanished boundary term  due to Lemma
\ref{lem17},  is bounded by
\begin{align*}
\int \left|\int \nm \cdot (\nabla v m  w_1\nu)\frac{\mu}{\nu}wdm \right|dx
&=  \int \left|\int \nabla v m  w_1\nu \cdot \nm \left(\frac{\mu}{\nu}w\right)dm \right| dx \\
& \leq
C_\ep |\nabla v|^2_0 | w_1|^2_{0,s} + \ep |\nm w|^2_{0,0}.
\end{align*}
Putting all together we have
\begin{eqnarray*}
\frac{d}{dt} |w|^2_{0,0} + \frac12 |\nm w|^2_{0,0} &\leq&
C(|v_2|^2_s+1)|w|^2_{0,0} + C|w_1|^2_{0,s}(|v|^2_0+ |\nabla
v|^2_0)\\
&\leq& C(A_1+1)|w|^2_{0,0}+ CA_2(|v|^2_0+ |\nabla v|^2_0).
\end{eqnarray*}
Substitution of the estimates of $|\nabla v|^2_0$ and $\frac{d}{dt} |v|^2_0$
in \eqref{4.18} gives
\begin{eqnarray}\label{4.17}
\frac{d}{dt} (|v|^2_0 + |w|^2_{0,0}) + \frac12 |\nm w|^2_{0,0} \leq D
(|v|^2_0 + |w|^2_{0,0}) + D |u|^2_0  + D |\tau|^2_0,
\end{eqnarray}
where $D$ is a large constant depending on $C,A_1,A_2$, for example we may choose
$$
D=C(A_1+1)(A_2+1).
$$
The Gronwall inequality gives
\begin{equation*}
\sup_t (|v|^2_0+|w|^2_{0,0}) + \frac12\int_0^{T^*}|\nm w|^2_{0,0}dt
\leq De^{DT}\int_0^{T^*} |u|^2_0+ |\tau|^2_0 dt
\end{equation*}
for any $0<T^*\leq T$. Due to the similar estimate for $\tau$ as
\eqref{tau}, the right hand side is bounded by
\begin{eqnarray*}
De^{DT}\left( T^* \sup_t |u|^2_0 + C_\ep T^* \sup_t
|\varpi|^2_{0,0} + \ep\int_0^{T^*} |\nm \varpi|^2_{0,0} dt \right).
\end{eqnarray*}
We choose $\ds \ep = \frac1{4De^{DT}}$, $\ds T^*=\frac12\min
\left\{T, \frac1{(C_\ep+1)De^{DT}} \right\}$ and redefine $T=T^*$  to obtain
\begin{equation}
 ||(v_2, w_2)-(v_1,w_1)||^2_\mathbf{M}= ||(v, w)||^2_\mathbf{M}
\leq \frac12 ||(u_2,\varpi_2)-(u_1,\varpi_1)||^2_\mathbf{M}.
\end{equation}
This shows that $\mathcal{F}$ has a fixed point $(v,w)$ in
$\mathbf{M}$, which is a solution to the coupled problem
\eqref{main1}. Since $\mathcal{F}(v,w)=(v,w)$, \eqref{ineq3} and
Theorem \ref{prop1} imply that  $(v,w)\in \mathbf{X}_\mu$.

The uniqueness follows from the same computation of estimates for
the contraction mapping. Let $(v_i,f_i\nu^{-1}) (i=1,2)$ be
solutions of the coupled problem \eqref{main1}. Then $v=v_2-v_1$
solves \eqref{nv.1} with $u_i=v_i$, $u=v$, and $\tau=\tau_2-\tau_1$
computed from $f_i$. $w=(f_2-f_1)\nu^{-1}$ also solves
\eqref{contraction_f} with $ w_1= f_1\nu^{-1}$. Similar to
\eqref{4.17}, we obtain
\begin{eqnarray*}
\frac{d}{dt} (|v|^2_0 + |w|^2_{0,0}) + \frac12 |\nm w|^2_{0,0} \leq D
(|v|^2_0 + |w|^2_{0,0} + |\tau|^2_0).
\end{eqnarray*}
It follows from the estimate for $\tau$ and Gronwall inequality that
$(v,w)\equiv (0,0)$,  which gives the uniqueness of  problem
\eqref{main1}.

\section{A further look at $b\geq 6$}\label{section6}

In this section, we sketch proofs of Theorem \ref{thm4} and Theorem
\ref{thm2} for the case of $\mu=\mu_0$.

Consider \eqref{FPE} when $x$ is not involved, i.e.,  \eqref{ff}.
The corresponding w-problem for $w=f\nu^{-1}-q$ with $\mu=\mu_0$
solves \eqref{ww} with the operator $L$ replaced by
\begin{equation}\label{L_0}
L_0[w]=-\frac12 \nabla \cdot (\nabla w \mu_0) + \left(2-\frac12 b
-\theta \right)m \cdot \nabla w \rho^{\theta-1}+ \nabla\cdot (\kappa  m w \mu_0) -
K_0 w,
\end{equation}
where
\begin{equation}\label{K_0_1}
K_0= \left[N(b/2-1)+2\kappa m\cdot m(1-\theta)\right]\rho^{\theta-1}.
\end{equation}
Define the conjugate of
$\mu_0$ as  \eqref{ms},  $\mu_0^*=\rho^{\theta-2}$,
then  $K_0$ can be rewritten as
\begin{equation}\label{K_0_2}
K_0= [N(b/2-1)+2\kappa m\cdot m(1-\theta)]\sqrt{\mu_0\mu_0^*}.
\end{equation}

To ensure well-posedness of \eqref{ww}, we need to check the
coercivity of $\mathcal{B}_0[w,w;t]$, which is defined as
\begin{align*}
\frac12 \int |\nabla w|^2\mu_0 dm & = \mathcal{B}_0[w,w;t] - \left(2-\frac12 b
-\theta \right)
\int m \cdot \nabla w w \rho^{\theta-1} dm \\
& \qquad -\int \nabla\cdot  (\kappa m w \mu_0)w dm + \int K_0 w^2dm.
\end{align*}
From the proof of Lemma \ref{lem Energy}, the last two terms are
bounded by
\begin{equation*}
C_\ep \int w^2\mu_0 dm + \ep \int |\nabla w|^2 \mu_0 dm,
\end{equation*}
where the embedding theorem \eqref{2.27} has been used. For small
enough $\ep$, this estimate yields
\begin{equation*}
\frac14\int |\nabla w|^2 \mu_0 dm \leq \mathcal{B}_0[w,w;t] + C \int
w^2 \mu_0 dm,
\end{equation*}
as long as
\begin{equation*}
\int \left(2-\frac12 b -\theta\right)m\cdot \nabla w
w\rho^{\theta-1} dm \geq 0,
\end{equation*}
for $w\in \Hgo$. This is indeed the case, as shown below.

\begin{lem}\label{lemma5-1}
Let $w\in \Hgo$. Then
\begin{equation}\label{5.1}
\int (2-\frac12 b -\theta)m \cdot \nabla w w\rho^{\theta-1} dm \geq 0.
\end{equation}
\end{lem}
\begin{proof}
 From  $-1<\theta<1$ and $b\geq 6$,  we see that $(2-b/2-\theta)<0$.
It suffices to show
$$\dis\int m\cdot \nabla w w\rho^{\theta-1} dm = \frac12 \int m \cdot \nabla w^2 \rho^{\theta-1} dm  \leq 0.
$$
Integration by parts gives
\begin{align*}
\int m \cdot \nabla w^2 \rho^{\theta-1} dm & = -\int w^2
(N\rho^{\theta-1}+2(1-\theta)|m|^2\rho^{\theta-2})dm +  \int_{\p B}
w^2 \rho^{\theta-1}m\cdot \frac{m}{|m|} dS \\
& \leq \sqrt{b}  \int_{\p B} w^2 \rho^{\theta-1} dS=0.
\end{align*}
Here we use the fact that $w^2 \rho^{\theta-1} \in W^{1,1}$ and $w^2
\rho^{\theta-1}|_{\p B}=0$.   To see this,  for any $w\in \Hgo$,  we
estimate
\begin{eqnarray*}
\int w^2 \rho^{\theta-1}+ |\nabla (w^2\rho^{\theta-1})| dm &\leq& \int w^2\rho^{\theta-1} + 2 |w\nabla w|\rho^{\theta-1} + 2(1-\theta)|m w^2| \rho^{\theta-2} dm\\
&\leq& C\int w^2 \sqrt{\mu_0\mu_0^*} + |w||\nabla
w|\sqrt{\mu_0\mu_0^*} + w^2\mu^*_0 dm\\
&\leq & C\|w\|^2_{H^1_{\mu_0}},
\end{eqnarray*}
due to the embedding theorem \eqref{2.27}.   Thus $w^2\rho^{\theta-1}|_{\p B} \in L^1(\p B)$ from the trace theorem and it is zero from the fact that $ C^\infty_c $ is
a dense subset of $\Hgo$. Thus \eqref{5.1}  follows.
\end{proof}
We now turn to the FPE problem including $x$-variable. The first
step in the proof of Theorem \ref{thm3} remains valid for
$\mu=\mu_0$. To check the second part of the proof, we need only
look at two extra terms beyond those in (\ref{lr}).
\begin{equation*} \dis -\left(2-\frac12 b -\theta\right)
\int m \cdot\nm \p^\gamma w \rho^{\theta-1} \p^\gamma w dm, \quad
-\left(2-\frac12 b -\theta\right) \int m \cdot\nm \p^\gamma q
\rho^{\theta-1} \p^\gamma w dm.
\end{equation*}
The first term is non-positive from Lemma \ref{lemma5-1}, and the
second term is bounded by
\begin{equation*}
C \left| \int \!\!\int m \cdot\nm \p^\gamma q \rho^{\theta-1}
\p^\gamma w dm \right| \leq C_\ep \int |\nm \p^\gamma q|^2 \mu_0 dm
+ \ep \int |\p^\gamma w|^2 \mu_0^* dm.
\end{equation*}
These ensure the same estimate  \eqref{5.2} and thus \eqref{estimate
f}.

For the well-posedness for the coupled problem, we utilize
$\theta<1$ and Lemma \ref{lemma5-1}. For example, for the proof of
Lemma \ref{lem2} with $\mu_0$
\begin{eqnarray*}
|\int \phi \nu \rho^{-1} dm|^2 = |\int \phi dm|^2 \leq \int \phi^2 \mu_0 dm \int \mu^{-1}_0 dm.
\end{eqnarray*}
Since $\theta<1$ we have  $\int \mu^{-1}_0 dm <\infty$, hence
\eqref{4.8}. Verification of other terms is omitted.

The remaining is to show Theorem \ref{thm2}, the solution $f$ is a
probability distribution if and only if $q|_{\p B}=0$ for
$\mu=\mu_0$, Positivity of $f$ follows as in Proposition
\ref{prop8}. For the conservation of mass, as in Proposition
\ref{prop9},  we only have to check \eqref{5.3}.

\begin{equation*}
\int_{B^\ep}(w\kappa m- \nabla w \cdot\nabla \phi_\ep \nu) dm -
\int_{B^\ep} w \rho^{b/2} \nabla \phi_\ep \cdot \nabla (\nu
\rho^{-b/2}) dm.
\end{equation*}
Since $\nu^2/\mu_0=\rho^{2-\theta}$ and $2-\theta >1$
\begin{equation*}
 \frac{\ep}{2} \int_{\p B_r} |\nabla \phi_\ep|^2 \rho^{2-\theta} dS
\end{equation*}
converges to $0$ as $\ep\to 0$. Thus the first term converges to $0$
as well. On the other hand, the same argument shows that the second
term converges to $\dis C\int_{\p B} q dS $ for some nonzero
constant C.  Hence, we conclude Theorem \ref{thm2} under the
assumption of Theorem \ref{thm4}.

\section{Conclusion}
In this paper, we have analyzed the FENE Dumbbell model which is of
bead-spring type Navier-Stokes-Fokker-Planck models for dilute
polymeric fluids, with our focus on developing a local
well-posedness theory subject to a class of Dirichlet-type boundary
conditions
$$
f\nu^{-1}=q \quad{\rm on} \; \p B
$$
for the polymer distribution $f$, where $\nu$ depends on $b>0$ through the distance function,
and $q$ is a given smooth function measuring the relative ratio of $f/\nu$ near boundary.
We have thus identified a sharp Dirichlet-type boundary
requirement for each $b>0$, while the sharpness of the boundary requirement is a consequence of the existence result
for each specification of the boundary behavior.
It has been shown that the probability density governed by the Fokker-Planck equation  approaches  zero near boundary,
necessarily faster than the distance function $d$ for $b>2$, faster than $d|ln d|$ for $b=2$,  and as fast as $d^{b/2}$ for $0<b<2$.
Moreover, the sharp boundary requirement for $b\geq 2$ is also sufficient for the distribution to remain a probability density.


\bigskip
\section*{Acknowledgments}  Liu's research was partially supported by the
National Science Foundation under Kinetic FRG grant DMS07-57227 and grant DMS09-07963.
The authors would like to thank Robert Pego for helpful discussions.


\begin{thebibliography}{10}
\bibitem{ACM09}
A. ~Arnold, J. A. ~Carrillo, and C. ~Manzini.
\newblock Refined long-time asymptotics for some polymeric fluid flow
models.
\newblock Preprint, 2009.

\bibitem{BaSchSu05}
J.~W. Barrett, C.~Schwab, and E.~S{\"u}li.
\newblock Existence of global weak solutions for some polymeric flow models.
\newblock {\em Math. Models Methods Appl. Sci.}, 15(6):939--983, 2005.

\bibitem{BaSu07}
J.~W. Barrett, and E.~S{\"u}li.
\newblock Existence of global weak solutions to kinetic models of dilute polymers.
\newblock {\em Multiscale Model. Simul.}, 6:506--546, 2007.

\bibitem{BaSu08}
J.~W. Barrett, and E.~S{\"u}li.
\newblock Existence of global weak solutions to dumbbell models for dilute polymers with microscopic
cut-off.
\newblock {\em Math. Mod. Meth. Appl. Sci. }, 18:935--971, 2008.

\bibitem{BaSu10}
J.~W. Barrett, and E.~S{\"u}li.
\newblock Existence of equilibration of  global weak solutions to finitely extensible nonlinear bead-spring chain models for dilute polymers.
\newblock {\em Preprint }, 2010.



\bibitem{BCAH87}
R.~B. Bird, C.~Curtiss, R.~C. Armstrong, and O.~Hassager.
\newblock {\em Dynamics of Polymeric Liquids, Volume 2: Kinetic Theory}.
\newblock Wiley Interscience, New York, 1987.

\bibitem{Ch09}
L.~Chupin.
\newblock The {FENE} model for viscoelastic thin film flows.
\newblock {\em Methods Appl. Anal.}, 16(2):217-261, 2009.

\bibitem{CL04}
C. Chauvi\`{e}re, and A.~ Lozinski.
\newblock Simulation of complex viscoelastic flows using
Fokker-Planck equation: 3D FENE model.
\newblock {\em J. Non-Newtonian Fluid Mech. }, 122:201--214, 2004.

\bibitem{CL04+}
C. Chauvi\`{e}re, and A.~ Lozinski.
\newblock Simulation of dilute polymer solutions using a
Fokker-Planck equation.
\newblock {\em J. Comput. Fluids.}, 33:687--696, 2004.

\bibitem{CFTZ07}
P.~Constantin, C.~Fefferman, E. S.~Titi and A. ~Zarnescu.
\newblock Regularity of coupled two-dimensional nonlinear
Fokker-Planck and Navier-Stokes systems.
\newblock {\em Comm. Math. Phys.}, 270(3):789–811, 2007.

\bibitem{CM01}
J.-Y. ~Chemin and N. ~Masmoudi.
\newblock About lifespan of regular solutions of equations related to viscoelastic
fluids.
\newblock {\em SIAM J. Math. Anal.}, 33(1):84–112, 2001.


\bibitem{CM08}
P.~Constantin and N. ~Masmoudi.
\newblock Global well-posedness for a Smoluchowski equation coupled with Navier-
Stokes equations in 2D.
\newblock {\em Comm. Math. Phys.}, 278(1):179–191, 2008.


\bibitem{DE:1986}
M.~Doi and S.~F. Edwards.
\newblock {\em The Theory of Polymer Dynamics}.
\newblock {Oxford University Press}, Oxford, 1986.


\bibitem{DL09}
P.~Degond and H. ~Liu.
\newblock Kinetic models for polymers with inertial effects.
\newblock {\em Netw. Heterog. Media}, 4(4):625–647, 2009.

\bibitem{DLP02}
P.~Degond, M.~Lemou and M.~Picasso.
\newblock Viscoelastic fluid models derived from kinetic equations for polymers.
\newblock {\em SIAM J. Appl. Math.}, 62(5):1501–1519, 2002.

\bibitem{DLY05}
Q.~Du, C.~Liu, and P.~Yu.
\newblock F{ENE} dumbbell model and its several linear and nonlinear closure
  approximations.
\newblock {\em Multiscale Model. Simul.}, 4(3):709--731, 2005.

\bibitem{ETiZh04}
W.~E, T.~Li, and P.~Zhang.
\newblock Well-posedness for the dumbbell model of polymeric fluids.
\newblock {\em Comm. Math. Phys.}, 248(2):409--427, 2004.





\bibitem{GS90}
C.~Guillop´e and J.~C. ~Saut.
\newblock Existence results for the flow of viscoelastic fluids with a differential constitutive
law.
\newblock {\em Nonlinear Anal.}, 15(9):849–869, 1990.

\bibitem{GS90+}
C. ~Guillop´e and J.~C. ~Saut.
\newblock Global existence and one-dimensional nonlinear stability of shearing motions
of viscoelastic fluids of Oldroyd type.
\newblock {\em RAIRO Mod´el. Math. Anal. Num´er.}, 24(3):369–401, 1990.

\bibitem{HZ09}
L. ~He and P. ~Zhang.
\newblock L2 decay of solutions to a micro-macro model for polymeric fluids near equilibrium.
\newblock {\em SIAM J. Math. Anal.}, 40(5):1905–1922, 2008/09.

\bibitem{JoLe03}
B.~Jourdain and T.~Leli{\`e}vre.
\newblock Mathematical analysis of a stochastic differential equation arising
  in the micro-macro modelling of polymeric fluids.
\newblock In {\em Probabilistic methods in fluids}, pages 205--223. World Sci.
  Publ., River Edge, NJ, 2003.

\bibitem{JLL04}
B.~Jourdain, T.~Leli{\`e}vre, and C.~Le~Bris.
\newblock Existence of solution for a micro-macro model of polymeric fluid: the
  {FENE} model.
\newblock {\em J. Funct. Anal.}, 209(1):162--193, 2004.

\bibitem{JLLO06}
B.~Jourdain, T.~Leli{\`e}vre, C.~Le~Bris, and F.~Otto.
\newblock Long-time asymptotics of amultiscale model for polymeric fluid flows.
\newblock {\em Arch. Ration. Mech. Anal.}, 181:97--148, 2006.

\bibitem{KS09}
D.~J. Knezevic, and E.~S{\"u}li.
\newblock Spectral Galerkin approximation of Fokker-Planck equations
with unbounded drift.
\newblock {\em ESAIM: M2AN}, 42(3):445--485, 2009.

\bibitem{Ku85}
A.~Kufner.
\newblock {\em Weighted {S}obolev Spaces}.
\newblock A Wiley-Interscience Publication. John Wiley \& Sons Inc., New York,
  1985.
\newblock Translated from the Czech.

\bibitem{KuPe03}
A. ~Kufner and L.-E. ~Persson.
\newblock Weighted inequalities of {H}ardy type.
\newblock World Scientific Publishing Co. Inc., River Edge, NJ, 2003

\bibitem{LiZhZh08}
F.~Lin, P.~Zhang, and Z.~Zhang.
\newblock On the global existence of smooth solution to the 2-{D} {FENE}
  dumbbell model.
\newblock {\em Comm. Math. Phys.}, 277(2):531--553, 2008.

\bibitem{LLZ07}
F.-H. Lin, C.~Liu, and P.~Zhang.
\newblock On a micro-macro model for polymeric fluids near equilibrium.
\newblock {\em Comm. Pure Appl. Math.}, 60(6):838--866, 2007.

\bibitem{LLZ05}
F.-H.~Lin, C. ~Liu and P. ~Zhang.
\newblock On hydrodynamics of viscoelastic fluids.
\newblock {\em Comm. Pure Appl. Math.}, 58(11):1437–1471, 2005.

\bibitem{LLZ08}
Z. ~Lei, C. ~Liu and Y. ~Zhou.
\newblock Global solutions for incompressible viscoelastic fluids.
\newblock {\em Arch. Ration. Mech. Anal.},  188(3):371--398, 2008.

\bibitem{LM2000}
P.-L.~Lions and N. ~Masmoudi.
\newblock Global solutions for some Oldroyd models of non-Newtonian flows.
\newblock {\em Chinese Ann. Math. Ser. B}, 21(2):131–146, 2000.

\bibitem{LeZ05}
Z.~Lei and Y. ~Zhou.
\newblock Global existence of classical solutions for the two-dimensional Oldroyd model via the
incompressible limit.
\newblock {\em SIAM J. Math. Anal.}, 37(3):797–814, 2005.

\bibitem{LiZh08}
F.~H. Lin and P.~Zhang.
\newblock The {FENE} dumbbell model near equilibrium.
\newblock {\em Acta Math. Sin. (Engl. Ser.)}, 24(4):529--538, 2008.

\bibitem{LiMa07}
P.-L. Lions and N.~Masmoudi.
\newblock Global existence of weak solutions to some micro-macro models.
\newblock {\em C. R. Math. Acad. Sci. Paris}, 345(1):15--20, 2007.

\bibitem{LiLi08}
C.~Liu and H.~Liu.
\newblock Boundary conditions for the microscopic {FENE} models.
\newblock {\em SIAM J. Appl. Math.}, 68(5):1304--1315, 2008.

\bibitem{LiSh08}
H.~Liu and J.~Shin.
\newblock Global well-posedness for the microscopic FENE model with a sharp boundary condition
\newblock {Preprint}, 2009.

\bibitem{LY10}
H.~Liu and H. ~Yu.
\newblock Entropy satisfying methods for the Fokker-Planck equation of {FENE} dumbbell model for polymers.
\newblock {In preparation}, 2010.


\bibitem{Ma08}
N.~Masmoudi.
\newblock Well-posedness for the {FENE} dumbbell model of polymeric flows.
\newblock {\em Comm. Pure Appl. Math.}, 61(12):1685--1714, 2008.

\bibitem{Ma10}
N.~Masmoudi.
\newblock Global existence of weak solutions to the {FENE} dumbbell model of polymeric flows.
\newblock {\em Preprint},  2010.


\bibitem{MB2002}
A.~J. ~Majda and A. ~L. ~Bertozzi.
\newblock Vorticity and incompressible flow.
\newblock {\em Cambridge University Press}, Cambridge,  2002.

\bibitem{MZZ08}
N.~Masmoudi, P.~Zhang and Z.~Zhang.
\newblock Global well-posedness for 2D polymeric fluid models and growth
estimate.
\newblock {\em  Phys. D}, 237(10-12):1663–1675, 2008.


\bibitem{OP02}
R.G.~Owens and T.N. ~ Phillips.
\newblock {\em Computational Rheology}.
\newblock Imperial College Press, London, 2002.


\bibitem{Oe:1996}
H.~{\"O}ttinger.
\newblock {\em Stochastic Processes in Polymeric Liquids}.
\newblock Springer-Verlag, Berlin and New York, 1996.

\bibitem{Rm91}
M.~Renardy.
\newblock An existence theorem for model equations resulting from kinetic
  theories of polymer solutions.
\newblock {\em SIAM J. Math. Anal.}, 22(2):313--327, 1991.

\bibitem{Sch09}
M. ~E. ~Schonbek.
\newblock Existence and decay of polymeric flows.
\newblock {\em SIAM J. Math. Anal.}, 41(2):564–587,2009.

\bibitem{SY10}
J. ~Shen and H.~J. Yu.
\newblock On the approximation of the Fokker-Planck equation of {FENE} dumbbell model, I: A new weighted formulation and an optimal {S}pectral-{G}alerkin algorithm in 2-D.
\newblock {\em Preprint},  2010.

\bibitem{ZhZh06}
H.~Zhang and P.~Zhang.
\newblock Local existence for the {FENE}-dumbbell model of polymeric fluids.
\newblock {\em Arch. Ration. Mech. Anal.}, 181(2):373--400, 2006.

\bibitem{ZZZ08}
L. ~Zhang, H.~Zhang and P.~Zhang.
\newblock Global existence of weak solutions to the regularized Hookean dumbbell model.
\newblock {\em Commun. Math. Sci.}, 6(1):85--124, 2008.


\end{thebibliography}
\end{document}